\documentclass[12pt]{amsart}
\usepackage{a4wide,enumerate,xcolor}
\usepackage{amsmath,graphicx}
\allowdisplaybreaks

\let\pa\partial
\let\na\nabla
\let\eps\varepsilon
\newcommand{\N}{{\mathbb N}}
\newcommand{\R}{{\mathbb R}}
\newcommand{\diver}{\operatorname{div}}

\newcommand{\dd}{\mathrm{d}}
\newcommand{\dx}{\mathrm{d}x}
\newcommand{\dt}{\mathrm{d}t}
\newcommand{\nD}{n_{\mathrm{Dir}}}
\newcommand{\pD}{p_{\mathrm{Dir}}}
\newcommand{\VD}{V_{\mathrm{Dir}}}
\newcommand{\WD}{W_{\mathrm{Dir}}}
\newcommand{\GD}{\Gamma_{\mathrm{Dir}}}
\newcommand{\GN}{\Gamma_{\mathrm{Neu}}}

\newtheorem{theorem}{Theorem}
\newtheorem{lemma}[theorem]{Lemma}

\newtheorem{remark}[theorem]{Remark}

\begin{document}

\title[Degenerate drift--diffusion systems for memristors]{Degenerate drift--diffusion systems for memristors}

\author[A. J\"ungel]{Ansgar J\"ungel}
\address{Institute of Analysis and Scientific Computing, Technische Universit\"at Wien,
Wiedner Hauptstra\ss e 8--10, 1040 Wien, Austria}
\email{juengel@tuwien.ac.at}

\author[M. Vetter]{Martin Vetter}
\address{Institute of Analysis and Scientific Computing, Technische Universit\"at Wien,
Wiedner Hauptstra\ss e 8--10, 1040 Wien, Austria}
\email{martin.vetter@tuwien.ac.at}

\date{\today}

\thanks{The authors acknowledge partial support from
the Austrian Science Fund (FWF), grants P33010 and F65.
This work has received funding from the European
Research Council (ERC) under the European Union's Horizon 2020 research and innovation programme, ERC Advanced Grant no.~101018153.}

\begin{abstract}
A system of degenerate drift--diffusion equations for the electron, hole, and oxygen vacancy densities, coupled to the Poisson equation for the electric potential, is analyzed in a three-dimensional bounded domain with mixed Dirichlet--Neumann boundary conditions. The equations model the dynamics of the charge carriers in a memristor device in the high-density regime. Memristors can be seen as nonlinear resistors with memory, mimicking the conductance response of biological synapses. The global existence of weak solutions and the weak--strong uniqueness property is proved. Thanks to the degenerate diffusion, better regularity results compared to linear diffusion can be shown, in particular the boundedness of the solutions.
\end{abstract}

\keywords{Drift--diffusion equations, degenerate diffusion, global existence analysis, weak--strong uniqueness, memristors, semiconductors, neuromorphic computing.}

\subjclass[2000]{35B45, 35B65, 35K51, 35K65, 35Q81.}

\maketitle


\section{Introduction}

A memristor is a nonlinear resistor with memory, which may be utilized as an artificial neuron in neuromorphic computing. Neuromorphic computing aims to create computers that behave like parts of the human brain \cite{IeAm20}. Here, we consider oxide-based memristors consisting of a thin titanium dioxide layer between two metal electrodes \cite{Mla19}. Besides the electrons and holes (defect electrons), also the oxygen vacancies act as charge carriers. When an electric field is applied, the oxygen vacancies drift and change the boundary between the low- and high-resistance layers. In this way, memristors are able to mimic the conductance response of synapses. Advantages of these devices are the low power consumption, short switching time, and its nanosize.

Memristor devices can be modeled by drift--diffusion equations for the densities of electrons $n(x,t)$, holes $p(x,t)$, and oxygen vacancies $D(x,t)$, coupled selfconsistently to the Poisson equation for the electric potential $V(x,t)$, where $x\in\R^3$ is the spatial variable and $t\ge 0$ is the time \cite{GSTD13,SBW09}. In low-density regimes, the (scaled) diffusion fluxes are given by $\na n$, $\na p$, and $\na D$, respectively. However, in the case of high densities, the nonlinear relation $\na n^{\alpha_n}$ with $\alpha_n=5/3$ has to be used for the diffusion flux (and similarly for holes and oxygen densities) \cite[Chap.~5]{Jue09}.

The existence analysis of the low-density three-species memristor drift--diffusion system was investigated in \cite{JJZ23}, and the two-species drift--diffusion equations in the high-density regime was studied in \cite{Jue94,Jue95}. Nonlinear diffusion fluxes were assumed in \cite{GaSk05}, but the assumptions do not fit into our framework. Up to our knowledge, the analysis of degenerate drift--diffusion equations for more than two species is missing in the literature. The work \cite{JJZ23} has proved the global existence of solutions to the low-density memristor drift--diffusion system with low regularity only, namely $\sqrt{n}$, $\sqrt{p}$, $\sqrt{D}\in W^{1,1}(\Omega)$. In this paper, we explore to what extent the degenerate diffusion allows us to improve the regularity of the solutions.

\subsection{Model equations}

The dynamics of the densities and electric potential is assumed to be given by the equations
\begin{align}
  \pa_t n &= \diver J_n, \quad J_n = \na n^{\alpha_n}-n\na V, \label{1.n} \\
  \pa_t p &= -\diver J_p, \quad J_p = -(\na p^{\alpha_p}+p\na V), \label{1.p} \\
  \pa_t D &= -\diver J_D, \quad J_D = -(\na D^{\alpha_D}+D\na V), \label{1.D} \\
  \lambda^2\Delta V &= n-p-D+A(x)\quad\mbox{in }\Omega,\ t>0, \label{1.V}
\end{align}
where $J_n$, $J_p$, and $J_D$ are the current densities of the electrons, holes, and oxygen densities, respectively, $\lambda>0$ is the (scaled) Debye length, and $A(x)$ is the given immobile acceptor doping density. Following \cite{SBW09}, we neglect recombination--generation terms. 

When the effective density of states in the conduction band is much larger than the doping concentration (high-density regime), the drift--diffusion model with Fermi--Dirac statistics can be approximated by equations \eqref{1.n}--\eqref{1.p} with $\alpha_n=\alpha_p=5/3$ \cite{Jue96}. Since we want to understand mathematically the gain of regularity, we allow for general exponents $\alpha_n,\alpha_p>1$. One may argue that the oxygen vacancies evolve not necessarily in a high-density regime. However, we cannot expect any gain of regularity if $\alpha_D=1$ (see \cite{JJZ23}). For this reason, we also choose $\alpha_D>1$. We discuss the case $\alpha_D=1$ in Remark \ref{rem.D}. Fermi--Dirac statistics need to be used also for the charge transport through ion channels when the number of states in the channel is of the same order as the particle numbers \cite{Lat09}. Thus, our results can also be applied to the charged particle transport in confined ion channels.

We impose physically motivated mixed Dirichlet--Neumann boundary conditions,
\begin{align}
  n = \nD, \quad p = \pD, \quad V = \VD &\quad\mbox{on }\GD,\ t>0, \label{1.Dbc} \\
  J_n\cdot\nu = J_p\cdot\nu = \na V\cdot\nu = 0 &\quad\mbox{on }\GN,\ t>0, \label{1.Nbc} \\
  J_D\cdot\nu = 0 &\quad\mbox{on }\pa\Omega,\ t>0, \label{1.JD}
\end{align}
and the initial conditions
\begin{equation}
  n(\cdot,0) = n_I, \quad p(\cdot,0) = p_I, \quad D(\cdot,0) = D_I\quad\mbox{in }\Omega. \label{1.ic}
\end{equation}
The boundary part $\GN$ models insulating boundary segments, while $\GD$ is the union of Ohmic contacts for the electron and hole densities and the applied voltage. These boundary conditions are typically used in the memristor literature \cite{GSTD13,SBW09}. They can be considered as first-order approximations from the semiconductor Boltzmann equation \cite{Pou91}. According to \cite{Yam95}, a second-order approximation leads to Robin-type conditions. The oxygen vacancies are supposed not to leave the semiconductor domain, which leads to Neumann conditions. 

\subsection{Mathematical difficulty}

The misfit of boundary conditions (mixed for the electron and hole densities and Neumann for the oxygen vacancy density) provides the main mathematical difficulty. To illustrate the problem, let the hole density be fixed and set $\nD=0$. Then, using $\log n$ and $\log D$ as test functions in the weak formulations of \eqref{1.n} and \eqref{1.D}, respectively, adding both equations, integrating by parts, and using the Poisson equation \eqref{1.V}, we find that
\begin{align}\label{1.aux}
  \frac{\dd}{\dt}&\int_\Omega\big(n(\log n-1)+D(\log D-1)\big)\dx
  + 4\int_\Omega(|\na\sqrt{n}|^2+|\na\sqrt{D}|^2)\dx \\
  &= \int_\Omega\na(n-D)\cdot\na V\dx \nonumber \\
  &= -\frac{1}{\lambda^2}\int_\Omega(n-D)(n-p-D+A)\dx
  + \int_{\GD}(n-D)\na V\cdot\nu \dx. \nonumber
\end{align}
The first term on the right-hand side can be bounded by $C\int_\Omega(n+D)\dx$, since $(n-D)^2\ge 0$ removes the quadratic terms, but the second term involves $\na V\cdot\nu$ on $\GD$, which cannot be easily bounded. Moreover, the monotonicity trick $(n-D)^2\ge 0$ cannot be applied for more than two species. 

This issue can be overcome by deriving first some estimates from the free energy (see below) and then to apply the Gagliardo--Nirenberg inequality; see \cite{BFS14,ChLu95,GlHu97,GlHu05}. However, this idea only works in two space dimensions. For the three-dimensional situation, the authors of \cite{ChLu95} assumed full elliptic regularity for the Poisson equation to achieve uniform $W^{1,\infty}(\Omega)$ estimates for the potential. This is only possible if the Dirichlet and Neumann boundary parts do not meet. In \cite{BFPR14}, no-flux boundary conditions are assumed for the densities and the Robin condition $\na V\cdot\nu + cV = \xi$ on $\pa\Omega$. Then the boundary term in \eqref{1.aux} can be handled and global existence in three space dimensions could be concluded. Finally, a combination of local $W^{1,q}(\Omega)$ regularity with $q>1$ and the $L^1\log L^1$ bound from \eqref{1.aux} has led to a global existence result \cite{JJZ23}, but with rather low regularity. To deal with the three-dimensional case and the degeneracy, we assume that there exists $r\ge 3$ such that
\begin{equation}\label{1.W13}
  \|\na V\|_{L^r(\Omega)}\le C\|n-p-D+A\|_{L^{3r/(3+r)}(\Omega)} + C
\end{equation}%
for some constant $C>0$ depending on the boundary data. This assumption is satisfied if the intersection of the Dirichlet and Neumann boundary behaves not ``too wildly''; see the discussion in Section \ref{sec.main}. Our global existence result holds for $r=3$, while we can prove the boundedness of solutions if $r>3$.

\subsection{Key ideas}

A priori estimates are derived from the free energy. Introduce the internal energies
\begin{align}\label{1.internal}
  h_n(n) = \frac{n(n^{\alpha_n-1}-\nD^{\alpha_n-1})}{\alpha_n-1},\quad
  h_v(v) = \frac{p(p^{\alpha_p-1}-\pD^{\alpha_p-1})}{\alpha_p-1}, \quad 
  h_D(D) = \frac{D^{\alpha_D}}{\alpha_D-1},
\end{align}
and the free energy as the sum of the internal energies and the electric energy,
\begin{align*}
  H[n,p,D] = \int_\Omega\bigg(h_n(n) + h_p(p) + h_D(D) + D\VD
  + \frac{\lambda^2}{2}|\na(V-\VD)|^2\bigg)\dx,
\end{align*}
where $V$ solves \eqref{1.V} with the boundary conditions in \eqref{1.Dbc}--\eqref{1.Nbc}. The additional term $D\VD$ compensates a contribution coming from the electric energy when computing the energy dissipation. A formal computation, made rigorous in Section \ref{sec.ex} on the level of approximate solutions, shows the free energy inequality
\begin{align}\label{1.ei}
  \frac{\dd}{\dt}H[n,p,D] + \int_\Omega\bigg(
  \frac{|J_n|^2}{n} + \frac{|J_p|^2}{p} + \frac{|J_D|^2}{D}\bigg)\dx
  \le C(\nD,\pD,\VD,T), \quad t\in(0,T),
\end{align}
which provides a priori estimates for $n^{\alpha_n}$, $p^{\alpha_p}$, and $D^{\alpha_D}$ in $L^\infty(0,T;L^1(\Omega))$ as well as for $J_n/\sqrt{n}$, $J_p/\sqrt{p}$, and $J_D/\sqrt{D}$ in $L^2(0,T;L^2(\Omega))$. Gradient bounds are derived from the Gagliardo--Nirenberg inequality and elliptic regularity \eqref{1.W13}. To highlight here the idea, we consider the equation for the electron density only, fixing $p$ and $D$:
\begin{align}
  \|\na n^{\alpha_n-1/2}\|_{L^2(\Omega)}
  &= \frac{\alpha_n}{\alpha_n-1/2}\bigg\|\frac{J_n}{\sqrt{n}}
  + \sqrt{n}\na V\bigg\|_{L^2(\Omega)} \label{1.nan} \\
  &\le C\bigg\|\frac{J_n}{\sqrt{n}}\bigg\|_{L^2(\Omega)}
  + C\|\sqrt{n}\|_{L^6(\Omega)}\|\na V\|_{L^3(\Omega)}. \nonumber
\end{align}
As the first term on the right-hand side is bounded (thanks to \eqref{1.ei}), we only need to estimate the second term. This is done by applying the Gagliardo--Nirenberg inequality for some $\theta\in[0,1]$ and using the bound for $n$ in $L^{\alpha_n}(\Omega)$ from \eqref{1.ei}:
\begin{align*}
  \|\sqrt{n}\|_{L^6(\Omega)} 
  &= \|n^{\alpha_n-1/2}\|_{L^{3/(\alpha_n-1/2)}(\Omega)}^{
  1/(2\alpha_n-1)}
  \le C\|\na n^{\alpha_n-1/2}\|_{L^2(\Omega)}^{\theta/(2\alpha_n-1)}
  \|n\|_{L^{\alpha_n}(\Omega)}^{(1-\theta)/2} + C \\
  &\le C\|\na n^{\alpha_n-1/2}\|_{L^2(\Omega)}^{\theta/(2\alpha_n-1)} 
  + C.
\end{align*}
In a similar way, exploiting elliptic regularity and applying the Gagliardo--Nirenberg inequality for some $\widetilde\theta\in[0,1]$ again,
\begin{align*}
  \|\na V\|_{L^3(\Omega)} &\le C\|n\|_{L^{3/2}(\Omega)} + C
  = \|n^{\alpha_n-1/2}\|_{L^{3/(2\alpha_n-1)}(\Omega)}^{
  1/(\alpha_n-1/2)} + C \\
  &\le C\|\na n^{\alpha_n-1/2}\|_{L^2(\Omega)}^{
  \widetilde\theta/(\alpha_n-1/2)} + C.
\end{align*}
Inserting both estimates into \eqref{1.nan} yields
$$
  \|\na n^{\alpha_n-1/2}\|_{L^2(\Omega)}
  \le C\|\na n^{\alpha_n-1/2}\|_{L^2(\Omega)}^{(\theta+2\widetilde\theta)
  /(2\alpha_n-1)} + C,
$$
which provides a gradient bound for $n^{\alpha_n-1/2}$ if the exponent on the right-hand side is smaller than one, which holds if and only if $\alpha_n>6/5$. Observe that this includes the physical value $\alpha_n=5/3$.

We obtain from the gradient bound an a priori estimate for $\na n^{\alpha_n}$ in $L^{1}(\Omega)$, from which we infer a bound for $\pa_t n$ in some Sobolev space. This allows us to apply the Aubin--Lions lemma to conclude the compactness of the sequence of approximate solutions whose limit is a solution to the original problem \eqref{1.n}--\eqref{1.ic}. 

\subsection{Main results}\label{sec.main}

We impose the following hypotheses:

\begin{itemize}
\item[(H1)] Domain: $T>0$, $\Omega\subset\R^3$ is a bounded domain with Lipschitz boundary and $\pa\Omega=\GD\cup\GN$ satisfies
$\GD\cap\GN=\emptyset$, $\GN$ is relatively open in $\pa\Omega$, and $\GD$ has positive measure.
\item[(H2)] Data: $A\in L^{\infty}(\Omega)$, $\nD,\pD,\VD\in W^{1,\infty}(\Omega)$ satisfy $\nD,\pD\ge 0$ in $\Omega$, and $n_I,p_I,D_I\in L^2(\Omega)$ satisfy $n_I,p_I,D_I\geq 0$ in $\Omega$.
\item[(H3)] Elliptic regularity: There exists $r\ge 3$ such that for all $f\in L^{3r/(3+r)}(\Omega)$, there exists $C>0$ such that the weak solution $V$ to 
\begin{equation}\label{1.ellip}
  \Delta V=f\quad\mbox{in }\Omega, \quad V=\VD\quad\mbox{on }\GD, \quad \na V\cdot\nu=0\quad\mbox{on }\GN 
\end{equation}
satisfies
$\|V\|_{W^{1,r}(\Omega)}\le C\|f\|_{L^{3r/(3+r)}(\Omega)}+C$. Note that $3r/(3+r)=3/2$ if $r=3$.
\end{itemize}

Let us discuss Hypotheses (H1)--(H3). Our results are also valid in $d$-dimensional domains with more restrictive bounds on the exponents $\alpha_v$ ($v=n,p,D$) depending on $d\ge 1$. We consider the case $d=3$ because of its physical relevance and to simplify the notation. Moreover, we may allow for time-dependent boundary data; see, e.g., \cite[Sec.~2]{DGJ97}. 

The most restrictive condition is Hypothesis (H3). Indeed, for general elliptic problems \eqref{1.ellip} with mixed boundary conditions, we can only expect solutions $V\in W^{1,r}(\Omega)$ for some $r>2$ \cite{Gro94}. Under some conditions on the Dirichlet and Neumann boundary parts (in particular, $\GD$ and $\GN$ intersect with an ``angle'' not larger than $\pi$; see \cite[Prop.~3.4]{DiRe15}), the regularity improves to $r>3$ \cite[Theorem 4.8]{DiRe15}. If the domain is a two-dimensional polygon, precise regularity results can be found in \cite{Gri85}. Shamir's counterexample in \cite{Sha68} shows that $r\ge 4$ cannot be expected, even if the domain and the data are smooth. Generally, Hypothesis (H3) for some $r>3$ is satisfied if $\GD$ and $\GN$ do not meet in a ``too wild'' manner; see the examples in \cite[Prop.~7.1]{HaRe09}.

We introduce some notation. We set $\Omega_T:=\Omega\times(0,T)$ and for $q\ge 1$,
$$
  \WD^{1,q}(\Omega):=\{u\in W^{1,q}(\Omega):u=0\mbox{ on }\GD\}.
$$
Moreover, we write 
$$
  \sum_{v=n,p,D}F(v):=F(n)+F(p)+F(D), \quad \sum_{v=n,p,D}F(\bar{v})=F(\bar{n})+F(\bar{p})+F(\bar{D}) 
$$
for arbitrary functions $F$. Constants $C>0$ in the following computations are generic and may change their value from line to line. 

\begin{theorem}\label{thm.ex}
Let Hypotheses (H1)--(H3) with $r=3$ hold and assume that $6/5<\alpha_n,\alpha_p,\alpha_D\le 2$. Then there exists a solution $(n,p,D,V)$ to \eqref{1.n}--\eqref{1.ic} satisfying $n,p,D\ge 0$ in $\Omega_T$ and
\begin{align*}
  & n^{\alpha_n},p^{\alpha_p},D^{\alpha_D}\in L^\infty(0,T;L^1(\Omega))
  \cap L^2(0,T;H^1(\Omega)), \\
  & \pa_t n,\pa_t p\in L^2(0,T;\WD^{1,4/3}(\Omega)'), \quad 
  \pa_t D\in L^2(0,T;W^{1,4/3}(\Omega)'),\\
  & n^{\alpha_n-1/2},p^{\alpha_p-1/2},D^{\alpha_D-1/2}\in 
  L^2(0,T;H^1(\Omega)), \\
  & V\in L^\infty(0,T;H^1(\Omega))\cap L^2(0,T;W^{1,3}(\Omega)).
\end{align*}
The fluxes satisfy the regularity
$J_v\in L^2(0,T;L^{2\alpha_v/(\alpha_v+1)}(\Omega))$ for $v=n,p,D$. 
Moreover, if $n_I^{\alpha_n-1},p_I^{\alpha_p-1},D_I^{\alpha_D-1}\in L^2(\Omega)$ holds, then
$$
  n^{\alpha_n-1},p^{\alpha_p-1},D^{\alpha_D-1}\in 
  L^2(0,T;H^1(\Omega)).
$$
\end{theorem}

The upper bound $\alpha_v\le 2$ for $v=n,p,D$ is needed to derive a priori estimates for $n,p,D$ in $W^{1,\alpha_v}(\Omega)$ for $3/2\le\alpha_v\le 2$; see the proof of Lemma \ref{lem.grad}. Solutions to the porous-medium equation with exponent $\alpha$ in the whole space possess the optimal regularity in $L^\alpha(0,T;W^{1,\alpha}(\R^d))$ under the condition $\alpha\le 2$ \cite[Lemma D.1]{Ges21}, which indicates that our upper bound $\alpha_v\le 2$ is optimal.

As mentioned before, the proof of Theorem \ref{thm.ex} is based on the free energy inequality \eqref{1.ei} and the elliptic regularity assumed in Hypothesis (H3). To make inequality \eqref{1.ei} rigorous, we introduce suitable cutoff functions with parameter $k\in\N$ that satisfy the chain rule. A Leray--Schauder fixed-point argument shows the existence of approximate weak solutions $(n_k,p_k,D_k,V_k)$. The limit $k\to\infty$ can be performed after deriving the uniform bounds sketched in the previous subsection, and the limit function turns out to be a weak solution to \eqref{1.n}--\eqref{1.ic}.

\begin{theorem}[Regularity]\label{thm.regul}
Let the assumptions of Theorem \ref{thm.ex} hold. If additionally
$\alpha_n,\alpha_p,$ $\alpha_D>\alpha^*:=(11+\sqrt{37})/14\approx 1.22$ and $n_I,p_I,D_I\in L^\infty(\Omega)$ hold, then the weak solution constructed in Theorem \ref{thm.ex} satisfies
$$
  n,p,D\in L^\infty(0,T;L^q(\Omega))\quad\mbox{for all }1\le q<\infty,
  \quad V\in L^\infty(0,T;W^{1,3}(\Omega)).
$$
Moreover, if additionally Hypothesis (H3) holds for some $r>3$, the regularity improves to
$$
  n,p,D\in L^\infty(0,T;L^\infty(\Omega)),
  \quad V\in L^\infty(0,T;W^{1,r}(\Omega)).
$$%
\end{theorem}
Bounded weak solutions to drift--diffusion systems were obtained in \cite{GaGr96} for two species and in \cite{BGN22} for multiple species, but the technique in the latter work seems to work only for linear diffusion. In two space dimensions, the solutions to the memristor model \eqref{1.n}--\eqref{1.ic} are bounded \cite{JJZ23}; also see \cite{GlHu97}. The restriction to two space dimensions comes from the regularity $V\in W^{1,q}(\Omega)$ with $q>2$, due to the mixed boundary conditions. Upper bounds for the densities to a two-species degenerate drift--diffusion model were found in \cite{Jue95} but under the assumption $V\in W^{2,q}(\Omega)$ for $q>3$. 

The first step of the proof of Theorem \ref{thm.regul} is an estimate for $n$ (and $p$, $D$) in $L^\infty(0,T;$ $L^{3/2}(\Omega))$. This follows from the energy inequality \eqref{1.ei} if $\alpha_n\ge 3/2$. If $\alpha_n<3/2$, we use an iteration argument, which seems to be new in this context. Assuming that $n$ is bounded in  $L^\infty(0,T;L^{\gamma_m+1}(\Omega))$, the aim is to derive a bound for $n$ in $L^\infty(0,T;L^{\gamma_{m+1}+1}(\Omega))$ for some $\gamma_{m+1}>\gamma_m$. It turns out that $(\gamma_m)$ satisfies a linear difference equation, whose solution satisfies $\gamma_m+1\to c(\alpha_n)$ as $m\to\infty$ for some $c(\alpha_n)>0$. The condition $\alpha_n>\alpha^*$ is necessary to ensure that $c(\alpha_n)\ge 3/2$, proving the claim $n\in L^\infty(0,T;L^{3/2}(\Omega))$. The second step of the proof is the derivation of $L^\infty(0,T;L^{\gamma+1}(\Omega))$ estimates for any $\gamma<\infty$ by choosing (a cutoff of) $n^\gamma-\nD^\gamma$ as a test function and applying the Gagliardo--Nirenberg inequality. 

Unfortunately, the $L^{\gamma+1}(\Omega)$ estimate depends on $\gamma$, and we cannot pass to the limit $\gamma\to \infty$ in this step. Therefore, we need slightly more regularity for the potential gradient in $L^r(\Omega)$ with $r>3$. This regularity allows us, in the third step, to apply an Alikakos-type iteration technique \cite{Ali79} which yields estimates for the densities in $L^{2^k}(\Omega)$ uniformly in $k\in\N$. The idea of the Alikakos method is to derive an estimate of the type
$$
  \|D\|_{L^{\gamma+1}(\Omega)}\le C + C\gamma^\beta\|D\|_{L^{(\gamma+1)/2}(\Omega)}\quad\mbox{for some }
  \beta>0.
$$
The halved exponent compensates the $\gamma$-dependent constant. In the degenerate case, the exponent is not halved, since we obtain
$$
  \|D\|_{L^{\gamma+1}(\Omega)}\le C + C\gamma^\beta\|D\|_{L^{(\gamma+\alpha_D)/2}(\Omega)}, 
  \quad\mbox{where }\alpha_D>1.
$$
We show that the Alikakos technique can be extended to the degenerate case. While the boundedness of solutions with linear diffusion was shown in two space dimensions, the degeneracy allows us to prove this result in three space dimensions. Theorem \ref{thm.regul} is the most original part of the paper.

\begin{theorem}[Weak--strong uniqueness]\label{thm.wsu}
Let the assumptions of Theorem \ref{thm.ex} hold. Let $(n,p,D,V)$ be a bounded weak solution to \eqref{1.n}--\eqref{1.ic}, satisfying the regularity stated in Theorem \ref{thm.ex}. Furthermore, let $(\bar{n},\bar{p},\bar{D},\bar{V})$ be a strong solution to \eqref{1.n}--\eqref{1.ic} in the sense that there exists $m>0$ such that $\bar{n},\bar{p},\bar{D}\ge m>0$ in $\Omega_T$ and
\begin{align*}
  & \bar{n},\bar{p},\bar{D}\in L^\infty(\Omega_T), \quad
  \pa_t\bar{n},\pa_t\bar{p}\in L^2(0,T;H_D^1(\Omega)'), 
  \quad \pa_t\bar{D}\in L^2(0,T;H^1(\Omega)'), \\
  & h_n'(\bar{n})-\bar{V},\ h_p'(\bar{p})+\bar{V},
  \ h_D'(\bar{D})+\bar{V}
  \in L^\infty(0,T;W^{2,\infty}(\Omega)).
\end{align*}
Then $(n,p,D,V)=(\bar{n},\bar{p},\bar{D},\bar{V})$ in $\Omega_T$.
\end{theorem}

The uniqueness of solutions to drift--diffusion equations is a delicate issue because of the simultaneous presence of degenerate diffusion and nonlinear drift. Often, uniqueness results need additional assumptions, like  boundedness of the fluxes \cite[Theorem 3.2]{GaGr89} or, in case of nonlinear diffusion fluxes, the regularity $V\in W^{1,q}(\Omega)$ for $q>d$ with $d$ being the space dimension; see \cite[Theorem 5.1]{Gaj94} and \cite[Theorem 6.1]{GaGr96}. The uniqueness of weak solutions in two dimensions was proved in \cite{GlHu97}, using the regularity $V\in W^{1,q}(\Omega)$ for some $q>2$. Uniqueness results for degenerate drift--diffusion equations under additional conditions have been proved in \cite{DGJ01,Jue97}. Therefore, we restrict ourselves to show the weak--strong uniqueness property.
Note that with the higher elliptic regularity $r>3$, the weak solutions constructed in Theorem~\ref{thm.regul} satisfy the assumptions of Theorem \ref{thm.wsu}.

The proof of this theorem is based on the relative free energy, which is defined by
\begin{equation}\label{1.relent}
  H[n,p,D|\bar{n},\bar{p},\bar{D}] = \int_\Omega\bigg(h_n(n|\bar{n})
  + h_p(p|\bar{p}) + h_D(D|\bar{D}) + \frac{\lambda^2}{2}
  |\na(V-\bar{V})|^2\bigg)\dx,
\end{equation}
where the relative entropy density is given by
\begin{equation}\label{1.rel}
  h_v(v|\bar{v}) = h_v(v) - h_v(\bar{v}) - h'_v(\bar{v})(v-\bar{v})
  \quad\mbox{and}\quad h_v(v) = \frac{v^{\alpha_v}}{\alpha_v-1}, 
  \quad v=n,p,D.
\end{equation}
Let $(n,p,D,V)$ and $(\bar{n},\bar{p},\bar{D},\bar{V})$ be two solutions to \eqref{1.n}--\eqref{1.ic} as described in Theorem \ref{thm.wsu}. A computation, detailed in Section \ref{sec.wsu}, shows that
\begin{align*}
  \frac{\dd}{\dt}&H[n,p,D|\bar{n},\bar{p},\bar{D}]
  + \sum_{v=n,p,D}\int_\Omega v
  \big|\na\big((h'_v(v)-V)-(h'_v(\bar{v})-\bar{V})\big)\big|^2\dx \\
  &\le C\sum_{v=n,p,D}\int_\Omega h_v(v|\bar{v})\dx+ C\sum_{v=n,p,D}\|\na(V-\bar{V})\|_{L^2(\Omega)}
  \|v-\bar{v}\|_{L^2(\Omega)},
\end{align*}
where $C>0$ depends on the $W^{2,\infty}(\Omega)$ norm of $h'_v(\bar{v})-\bar{V}$. Since $(n,p,D)$ is assumed to be bounded, the inequality $(v-\bar{v})^2\le Ch_v(v|\bar{v})$ holds for $v=n,p,D$ (see \eqref{4.relentL2}). We conclude from Young's inequality that
$$
  \frac{\dd}{\dt}H[n,p,D|\bar{n},\bar{p},\bar{D}]
  \le CH[n,p,D|\bar{n},\bar{p},\bar{D}],
$$
and since $(n,p,D)$ and $(\bar{n},\bar{p},\bar{D})$ have the same initial data, Gronwall's lemma implies that both solutions coincide, proving the theorem.

\begin{remark}\rm
Our results are valid for an arbitrary number of charged particle species, like in ion transport. In this situation, the equations for the charge densities $u_i$ are
$$
  \pa_t u_i = \diver(\na u_i^{\alpha_i}+u_iz_i\na V), \quad
  i=1,\ldots,n, \quad \lambda^2\Delta V = \sum_{i=1}^n z_iu_i + A(x),
$$
where $z_i\in\R$ are the ionic charges, the exponents $\alpha_i>1$ satisfy the conditions imposed in the theorems, and initial and mixed boundary conditions are chosen. The reason that the results are valid for such systems is that we use the Poisson equation only through the $L^q(\Omega)$ norm of $\na V$ so that the drift terms can be handled as in the following sections.
\qed\end{remark}

The paper is organized as follows. Theorem \ref{thm.ex} is proved in Section \ref{sec.ex}. The regularity results of Theorem \ref{thm.regul} are shown in Section \ref{sec.regul}, and the weak--strong uniqueness property of Theorem \ref{thm.wsu} is proved in Section \ref{sec.wsu}.


\section{Existence of solutions}\label{sec.ex}

The aim of this section is to prove Theorem \ref{thm.ex}. We solve system \eqref{1.n}--\eqref{1.ic} by truncating the nonlinearities similarily as in \cite{JJZ23} but with a slightly different truncation. The existence of approximate solutions, based on the Leray--Schauder fixed-point theorem, is analogous to the one in \cite{JJZ23}. The approximate free energy inequality, similar to \eqref{1.ei}, is independent of the truncation parameter $k\in\N$. After deriving further uniform bounds, we apply the Aubin--Lions compactness lemma to pass to the limit $k\to\infty$ and obtain the existence of a solution to \eqref{1.n}--\eqref{1.ic}.

\subsection{Truncated system}

Let $k\in\N$, $k\ge 2$, and set
$$
  T_k(v) := \min\{k,\max\{k^{-1},v\}\}\in[k^{-1},k]\quad\mbox{for }
  v\in\R.
$$
We consider the regularized problem
\begin{align}
  \pa_t n_k &= \diver\big(\alpha_n T_k(n_k)^{\alpha_n-1}\na n_k
  - T_k(n_k)\na V_k\big), \label{2.nk} \\
  \pa_t p_k &= \diver\big(\alpha_p T_k(p_k)^{\alpha_p-1}\na p_k
  + T_k(p_k)\na V_k\big), \label{2.pk} \\
  \pa_t D_k &= \diver\big(\alpha_D T_k(D_k)^{\alpha_D-1}\na D_k
  + T_k(D_k)\na V_k\big), \label{2.Dk} \\
  \lambda^2\Delta V_k &= n_k-p_k-D_k+A(x)\quad\mbox{in }\Omega,\ t>0,
  \label{2.Vk}
\end{align}
subject to the initial conditions and mixed boundary conditions
\begin{align}
  n_k(\cdot,0)=n_I, \quad p_k(\cdot,0)=p_I, \quad D_k(\cdot,0)=D_I
  &\quad\mbox{in }\Omega, \label{2.ic} \\
  n_k=\nD, \quad p_k=\pD, \quad V_k=\VD
  &\quad\mbox{on }\GD,\ t>0, \label{2.Dbc} \\
  \na n_k\cdot\nu = \na p_k\cdot\nu = \na V_k\cdot\nu = 0
  &\quad\mbox{on }\GN,\ t>0, \label{2.Nbc} \\
  \na  D_k\cdot\nu = 0 &\quad\mbox{on }\pa\Omega,\ t>0. \label{2.JD}
\end{align}

\begin{lemma}\label{lem.exk}
Let Hypotheses (H1)--(H3) hold. Then there exists a weak solution $(n_k,p_k,D_k,$ $V_k)$ to \eqref{2.nk}--\eqref{2.JD} satisfying 
\begin{align*}
  & n_k,\,p_k,\,D_k,\,\in L^2(0,T;H^1(\Omega)), \quad
  V_k\in L^2(0,T;H^1(\Omega)), \\
  & \pa_t n_k,\,\pa_t p_k\in L^2(0,T;H^1_D(\Omega)'), \quad\pa_t D_k\in L^2(0,T;H^1(\Omega)').
\end{align*}
\end{lemma}

\begin{proof}
The proof is analogous to the proof of Lemma 2.1 in \cite{JJZ23} with the difference that we use the strictly positive cutoff $T_k(v)\ge k^{-1}>0$ and that equations \eqref{2.nk}--\eqref{2.Dk} are nonlinear in the diffusion term. However, since the truncated diffusion coefficients are strictly positive and bounded, the proof still applies. Compared to \cite{JJZ23}, we cannot conclude that $n_k$, $p_k$, and $D_k$ are nonnegative. 
\end{proof}

%

\subsection{Auxiliary functions}

For the derivation of uniform estimates, we need some auxiliary functions, which preserve the free energy structure and involve the cutoff $T_k$. Let $\gamma>1$, $v\in\R$ and introduce the functions
\begin{align*}
  & S_k^{\gamma-1}(v) = (\gamma-1)\int_0^v T_k(y)^{\gamma-2}\dd y, \quad
  S_k^0(v) = \int_0^v\frac{\dd y}{T_k(y)}, \quad
  R_k^\gamma(v) = \gamma\int_0^v S_k^{\gamma-1}(y)\dd y.
\end{align*}
These functions are constructed in such a way that the chain rules
\begin{equation}\label{2.chain}
\begin{aligned}
  & \na S_k^{\gamma-1}(v) = (\gamma-1)T_k(v)^{\gamma-2}\na v, \quad\na S_k^0(v) = \frac{\na v}{T_k(v)}, \quad
  \pa_t R_k^\gamma(v) = \gamma S_k^{\gamma-1}(v)\pa_t v
\end{aligned}
\end{equation}
hold for suitable smooth functions $v$. The functions $(S_k^{\gamma-1},S_k^0,R_k^\gamma)$ approximate $(v^{\gamma-1},\log v,$ $v^\gamma)$.
They satisfy the following inequalities.

\begin{lemma}\label{lem.RST}
There exists $C>0$ such that for sufficiently large $k\in\N$ and for all $v\in\R$,
\begin{align}
  T_k(v)^\gamma &\le S_k^\gamma(v) + C \quad\mbox{for }\gamma>0, \label{2.TS} \\
  (S_k^\gamma(v))^{\beta/\gamma} &\le CR_k^\beta(v) + C\quad\mbox{for }
  \beta>1,\,\gamma\ge \beta/2, \label{2.SR} \\
  v &\le CS_k^\beta(v)^{1/\beta} + C\quad\mbox{for }v\ge 0\mbox{ and }
  0<\beta\le 1. \label{2.vS}
\end{align}
Furthermore, for any $\delta>0$, there exists $C(\delta)>0$ such that for $\beta>1$ and $v\ge 0$,
\begin{equation}\label{2.vR}
  v\le \delta R^\beta_k(v) + C(\delta).
\end{equation}
\end{lemma}

\begin{proof}
The inequalities can be proved by elementary computations using the explicit expressions
\begin{align*}
  T_k(v)^\alpha = k^{-\alpha}, \quad 
  S_k^\alpha(v) = \alpha k^{1-\alpha}v, \quad
  R_k^\alpha(v) = \frac12\alpha(\alpha-1)k^{2-\alpha}v^2
\end{align*} 
for $v\le 1/k$;
\begin{align*}
  T_k(v)^\alpha &= v^\alpha, \quad
  S_k^\alpha(v) = v^\alpha + (\alpha-1)k^{-\alpha}, \\
  R_k^\alpha(v) &= v^\alpha + \alpha(\alpha-2)k^{1-\alpha}v
  - \frac12(\alpha-1)(\alpha-2)k^{-\alpha}
\end{align*}
for $1/k\le v\le k$;
\begin{align*}
  T_k(v)^\alpha &= k^\alpha, \quad
  S_k^\alpha(v) = \alpha k^{\alpha-1}v - (\alpha-1)
  (k^\alpha-k^{-\alpha}), \\
  R_k^\alpha(v) &= \frac12\alpha(\alpha-1)k^{\alpha-2}v^2
  - (\alpha-2)(k^{\alpha-1}-k^{1-\alpha}) \\
  &\phantom{xx}
  - \frac12(\alpha-1)(\alpha-2)(k^\alpha+k^{-\alpha}-2k^{2-\alpha})
\end{align*}
for $v\ge k$. We leave the details to the reader. For instance, inequality \eqref{2.vR} follows from the fact that $R_k^{\beta}$ grows at least like $\min\{2,\beta\}>1$.
\end{proof}

\subsection{Uniform estimates}

We proceed by deriving some estimates uniformly in $k$. Let $(n_k,p_k,D_k,V_k)$ be a weak solution to \eqref{2.nk}--\eqref{2.JD} according to Lemma \ref{lem.exk}. We define the truncated free energy by
\begin{align*}
  H_k[n_k,p_k,D_k] = \int_\Omega\bigg(h_{n,k}(n_k)
  + h_{k,p}(p_k) + h_{k,D}(D_k) + D_k\VD 
  + \frac{\lambda^2}{2}|\na(V_k-\VD)|^2\bigg)\dx,
\end{align*}
where the approximate internal energies are given by
\begin{align*}
  h_{n,k}(n_k) &= (\alpha_n-1)^{-1}\big(R_k^{\alpha_n}(n_k)
  - \alpha_n S_k^{\alpha_n-1}(\nD)n_k\big), \\
  h_{p,k}(p_k) &= (\alpha_p-1)^{-1}\big(R_k^{\alpha_p}(p_k)
  - \alpha_p S_k^{\alpha_p-1}(\pD)p_k\big), \\
  h_{D,k}(D_k) &= (\alpha_D-1)^{-1}R_k^{\alpha_D}(D_k).
\end{align*}

\begin{lemma}[Free energy inequality with cutoff]\label{lem.eik}
There exists a constant $C>0$, depending on the initial and boundary data but not on $k$, such that for $t>0$,
\begin{align*}
  H_k[n_k(t),p_k(t),D_k(t)]
  &+ \frac12\int_0^t\int_\Omega T_k(n_k)\bigg|\na\bigg(
  \frac{\alpha_n}{\alpha_n-1}S_k^{\alpha_n-1}(n_k)-V_k\bigg)\bigg|^2
  \dx\dd s \\
  &+ \frac12\int_0^t\int_\Omega T_k(p_k)\bigg|\na\bigg(
  \frac{\alpha_p}{\alpha_p-1}S_k^{\alpha_p-1}(p_k)+V_k\bigg)\bigg|^2
  \dx\dd s \\
  &+ \int_0^t\int_\Omega T_k(D_k)\bigg|\na\bigg(
  \frac{\alpha_D}{\alpha_D-1}S_k^{\alpha_D-1}(D_k)+V_k\bigg)\bigg|^2
  \dx\dd s
  \le C.
\end{align*}
\end{lemma}

\begin{proof}
The chain rule \eqref{2.chain} leads to
\begin{align*}
  \pa_t h_{n,k}(n_k) &= \frac{\alpha_n}{\alpha_n-1}
  \big\langle\pa_t n_k,
  S_k^{\alpha_n-1}(n_k)-S_k^{\alpha_n-1}(\nD)\big\rangle,
\end{align*}
and similarly for $h_{p,k}$ and $h_{D,k}$. Moreover, we have
$$
  \frac{\lambda^2}{2}\int_\Omega|\na(V-V_k)|^2\dx\Big|_0^t
  = -\int_0^t\big\langle\pa_t(n_k-p_k-D_k),V_k-\VD\big\rangle\dd s.
$$
This implies that
\begin{align*}
  H_k[n_k,p_k,D_k]\Big|_0^t 
  &= \int_0^t\frac{\dd}{\dt}H_k[n_k,p_k,D_k]\dd s \\
  &= \int_0^t\bigg\langle
  \pa_t n_k,\frac{\alpha_n}{\alpha_n-1}\big(S_k^{\alpha_n-1}(n_k)
  - S_k^{\alpha_n-1}(\nD)\big) - (V_k-\VD)\bigg\rangle\dd s \\
  &\phantom{xx}- \int_0^t\bigg\langle
  \pa_t p_k,\frac{\alpha_p}{\alpha_p-1}\big(S_k^{\alpha_p-1}(p_k)
  - S_k^{\alpha_p-1}(\pD) + (V_k-\VD)\bigg\rangle\dd s \\
  &\phantom{xx}- \int_0^t\bigg\langle
  \pa_t D_k,\frac{\alpha_D}{\alpha_D-1}S_k^{\alpha_D-1}(D_k)
  + V_k\bigg\rangle\dd s.
\end{align*}
Let us consider the first term on the right-hand side. We insert the drift--diffusion equation \eqref{2.nk} and use the chain rule \eqref{2.chain} as well as Young's inequality:
\begin{align*}
  \int_0^t&\bigg\langle
  \pa_t n_k,\frac{\alpha_n}{\alpha_n-1}\big(S_k^{\alpha_n-1}(n_k)
  - S_k^{\alpha_n-1}(\nD)\big) - (V_k-\VD)\bigg\rangle\dd s \\
  &= -\int_0^t\int_\Omega T_k(n_k)\bigg|\na\bigg(
  \frac{\alpha_n}{\alpha_n-1}S_k^{\alpha_n-1}(n_k)-V_k\bigg)\bigg|^2
  \dx\dd s \\
  &\phantom{xx}+ \int_0^t\int_\Omega T_k(n_k)\na\bigg(
  \frac{\alpha_n}{\alpha_n-1}S_k^{\alpha_n-1}(n_k)-V_k\bigg)\cdot
  \na\bigg(\frac{\alpha_n}{\alpha_n-1}
  S_k^{\alpha_n-1}(\nD)-\VD\bigg)\dx\dd s \\
  &\le -\frac12\int_0^t\int_\Omega T_k(n_k)\bigg|\na\bigg(
  \frac{\alpha_n}{\alpha_n-1}S_k^{\alpha_n-1}(n_k)-V_k\bigg)\bigg|^2
  \dx\dd s \\
  &\phantom{xx}+ \frac12\bigg\|\na\bigg(\frac{\alpha_n}{\alpha_n-1}
  S_k^{\alpha_n-1}(\nD)-\VD\bigg)\bigg\|_{L^\infty(\Omega)}^2
  \int_0^t\int_\Omega T_k(n_k)\dx\dd s.
\end{align*}
By assumption, the $W^{1,\infty}(\Omega)$ norms of $\nD$ and $\VD$ are finite, so the factor of the last integral is bounded. Treating the terms involving $\pa_t p_k$ and $\pa_t D_k$ in a similar way, we end up with
\begin{align}
  H_k[n_k,p_k,D_k,V_k]\Big|_0^t &\le
  -\frac12\int_0^t\int_\Omega T_k(n_k)\bigg|\na\bigg(
  \frac{\alpha_n}{\alpha_n-1}S_k^{\alpha_n-1}(n_k)-V_k\bigg)\bigg|^2
  \dx\dd s \label{2.Hk} \\
  &\phantom{xx}- \frac12\int_0^t\int_\Omega T_k(p_k)\bigg|\na\bigg(
  \frac{\alpha_p}{\alpha_p-1}S_k^{\alpha_p-1}(p_k)+V_k\bigg)\bigg|^2
  \dx\dd s \nonumber \\
  &\phantom{xx}- \int_0^t\int_\Omega T_k(D_k)\bigg|\na\bigg(
  \frac{\alpha_D}{\alpha_D-1}S_k^{\alpha_D-1}(D_k)+V_k\bigg)\bigg|^2
  \dx\dd s \nonumber \\
  &\le C\int_0^t\int_\Omega\big(T_k(n_k)+T_k(p_k)+T_k(D_k)\big)\dx\dd s,
  \nonumber 
\end{align}
Notice that we do not need to apply Young's inequality to the term involving $D_k$ since the no-flux boundary conditions directly allow for an integration by parts. Therefore, the dissipation term for $D_k$ has no factor $1/2$.

The right-hand side of \eqref{2.Hk} can be estimated by using Lemma \ref{lem.RST}. Indeed, we find that
$$
  T_k(n_k)\le T_k(n_k)^{\alpha_n} + C
  \le S_k^{\alpha_n}(n_k) + C \le CR_k^{\alpha_n}(n_k) + C,
$$
and similar for the other terms. This shows that
$$
  H_k[n_k,p_k,D_k,V_k]\Big|_0^t \le C\int_0^t H_k[n_k,p_k,D_k,V_k]
  \dd s + C,
$$
and Gronwall's lemma implies that $H_k[n_k,p_k,D_k,V_k](t)$ is bounded for all $t>0$. We deduce from this information that the right-hand side of \eqref{2.Hk} is bounded, thus finishing the proof.
\end{proof}

The previous lemma implies the following uniform bounds.

\begin{lemma}\label{lem.ener}
There exists $C>0$ independent of $k$ such that
\begin{align*}
  \|V_k\|_{L^\infty(0,T;H^1(\Omega))} &\le C, \\
  \sup_{0<t<T}\big(\|R_k^{\alpha_n}(n_k(t))\|_{L^1(\Omega)}
  + \|R_k^{\alpha_p}(p_k(t))\|_{L^1(\Omega)}
  + \|R_k^{\alpha_D}(D_k(t))\|_{L^1(\Omega)}\big) &\le C, \\
  \sup_{0<t<T}\big(\|T_k(n_k(t))\|_{L^{\alpha_n}(\Omega)}
  + \|T_k(p_k(t))\|_{L^{\alpha_p}(\Omega)}
  + \|T_k(D_k(t))\|_{L^{\alpha_D}(\Omega)}\big) &\le C, \\
  \big\|T_k(n_k)^{1/2}\big(\alpha_n T_k(n_k)^{\alpha_n-2}\na n_k
  - \na V_k\big)\big\|_{L^2(\Omega_T)} &\le C, \\
  \big\|T_k(p_k)^{1/2}\big(\alpha_p T_k(p_k)^{\alpha_p-2}\na p_k
  + \na V_k\big)\big\|_{L^2(\Omega_T)} &\le C, \\
  \big\|T_k(D_k)^{1/2}\big(\alpha_D T_k(D_k)^{\alpha_D-2}\na D_k
  + \na V_k\big)\big\|_{L^2(\Omega_T)} &\le C.
\end{align*}
\end{lemma}

\begin{proof}
The first and the last three bounds follow directly from Lemma \ref{lem.eik} by observing that the chain rules \eqref{2.chain} give
$$
  \na\bigg(\frac{\alpha_n}{\alpha_n-1}S_k^{\alpha_n-1}(n_k)-V_k\bigg)
  = \alpha_n T_k(n_k)^{\alpha_n-2}\na n_k - \na V_k.
$$
The bounds on $R_k$ are a consequence of the definition of the approximate internal energies and Lemma \ref{lem.RST}. Indeed, by definition of $h_{k,n}(n_k)$,
$$
  \int_\Omega R_k^{\alpha_n}(n_k(t))\dx
  = (\alpha_n-1)\int_\Omega h_{k,n}(n_k(t))\dx 
  + \alpha_n\int_\Omega S_k^{\alpha_n-1}(\nD)n_k(t)\dx.
$$
For $n_k\le 0$, the last term is nonpositive so that, by Lemma \ref{lem.eik},
$$
  \int_\Omega R_k^{\alpha_n}(n_k(t))\mathrm{1}_{\{n_k\le 0\}}\dx \le C.
$$
On the other hand, for $n_k>0$, we apply \eqref{2.vR} to find, for any $\delta>0$, that
$$
  \int_\Omega R_k^{\alpha_n}(n_k(t))\mathrm{1}_{\{n_k>0\}}\dx
  \le C(\delta) + \delta C(\nD)\int_\Omega R_k^{\alpha_n}(n_k(t))
  \mathrm{1}_{\{n_k>0\}}\dx,
$$
and for sufficiently small $\delta>0$, the last term can be absorbed by the left-hand side. The remaining estimate for $T_k(n_k)$ is a consequence of the previous bound and estimate $T_k(n_k)^{\alpha_n}\le CR_k^{\alpha_n}(n_k)+C$. Similar estimates hold for $p_k$ and $D_k$.
\end{proof}

Next, we derive some gradient bounds for the approximate densities. This is the key lemma of the existence analysis.

\begin{lemma}[Gradient bounds]\label{lem.naSk}
Let $\alpha_n,\alpha_p,\alpha_D>6/5$. Then there exists $C>0$ independent of $k$ such that
\begin{align*}
  \|\na S_k^{\alpha_n-1/2}(n_k)\|_{L^2(\Omega_T)}
  + \|\na S_k^{\alpha_p-1/2}(p_k)\|_{L^2(\Omega_T)}
  + \|\na S_k^{\alpha_D-1/2}(D_k)\|_{L^2(\Omega_T)} \le C.
\end{align*}
In particular, we have a uniform bound for $S_k^{\alpha_v-1/2}(v_k)$ in $L^2(0,T;H^1(\Omega))$ for $v=n,p,D$.
\end{lemma}

\begin{proof}
It follows from the chain rules \eqref{2.chain} and the energy estimates of Lemma \ref{lem.ener} that
\begin{align}
  \|\na &S_k^{\alpha_n-1/2}(n_k)\|_{L^2(\Omega)}
  = \bigg(\alpha_n-\frac12\bigg)\|T_k(n_k)^{\alpha_n-3/2}\na n_k\|_{L^2(\Omega)} \label{2.naSk} \\
  &\le \frac{\alpha_n-1/2}{\alpha_n}\big\|T_k(n_k)^{1/2}\big(
  \alpha_n T_k(n_k)^{\alpha_n-2}\na n_k-\na V_k\big)
  \big\|_{L^2(\Omega)} \nonumber \\
  &\phantom{xx}
  + \frac{\alpha_n-1/2}{\alpha_n}\|T_k(n_k)^{1/2}\na V_k\|_{L^2(\Omega)}
  \le C + C\|T_k(n_k)^{1/2}\|_{L^6(\Omega)}\|\na V_k\|_{L^3(\Omega)}.
  \nonumber
\end{align}
We estimate the $L^6(\Omega)$ norm of $T_k(n_k)^{1/2}$ by using \eqref{2.TS} and the Gagliardo--Nirenberg inequality:
\begin{align*}
  \|T_k(n_k)^{1/2}\|_{L^6(\Omega)}
  &= \|T_k(n_k)^{\alpha_n-1/2}\|_{L^{3/(\alpha_n-1/2)}
  (\Omega)}^{1/(2\alpha_n-1)}
  \le \|S_k^{\alpha_n-1/2}(n_k)\|_{L^{3/(\alpha_n-1/2)}
  (\Omega)}^{1/(2\alpha_n-1)} + C \\
  &\le C\|\na S_k^{\alpha_n-1/2}(n_k)\|_{L^2
  (\Omega)}^{\theta(\alpha_n)/(2\alpha_n-1)}
  \|S_k^{\alpha_n-1/2}(n_k)\|_{L^{\alpha_n/(\alpha_n-1/2)}
  (\Omega)}^{(1-\theta(\alpha_n))/(2\alpha_n-1)} + C,
\end{align*}
where 
$$
  \theta(\alpha_n) 
  = \frac{(2\alpha_n-1)(3-\alpha_n)}{5\alpha_n-3}\in[0,1]
$$
(this only requires that $1\le\alpha_n\le 3$). We deduce from \eqref{2.SR} with $\beta=\alpha_n$ and $\gamma=\alpha_n-1/2$ that
\begin{align}\label{2.estS}
  \sup_{t\in(0,T)}
  \|S_k^{\alpha_n-1/2}(n_k(t))\|_{L^{\alpha_n/(\alpha_n-1/2)}
  (\Omega)}^{\alpha_n/(\alpha_n-1/2)}
  \le C\sup_{t\in(0,T)}\int_\Omega R_k^{\alpha_n}(n_k(t))\dx + C \le C,
\end{align} 
where the last inequality follows from Lemma \ref{lem.ener}. 
This implies that
$$
  \|T_k(n_k)^{1/2}\|_{L^6(\Omega)} 
  \le C\|\na S_k^{\alpha_n-1/2}(n_k)\|_{L^2(\Omega)}^{
  \theta(\alpha_n)/(2\alpha_n-1)} + C.
$$
Similar estimates hold for $T_k(p_k)$ and $T_k(D_k)$:
\begin{align*}
  \|T_k(p_k)^{1/2}\|_{L^6(\Omega)} &\le C\|\na S_k^{\alpha_p-1/2}(p_k)\|_{L^2(\Omega)}^{
  \theta(\alpha_p)/(2\alpha_p-1)} + C, \\
  \|T_k(D_k)^{1/2}\|_{L^6(\Omega)} &\le C\|\na S_k^{\alpha_D-1/2}(D_k)\|_{L^2(\Omega)}^{
  \theta(\alpha_D)/(2\alpha_D-1)} + C.
\end{align*}
The function $\alpha\mapsto\theta(\alpha)$ is decreasing for $\alpha>1$. Hence, $\theta(\alpha_0)$ with $\alpha_0=\min\{\alpha_n,\alpha_p,\alpha_D\}$ is larger than $\theta(\alpha_v)$ for $v=n,p,D$, and collecting the previous estimates, we obtain
\begin{align}
  \|T_k&(n_k)^{1/2}\|_{L^6(\Omega)} 
  + \|T_k(p_k)^{1/2}\|_{L^6(\Omega)}
  + \|T_k(D_k)^{1/2}\|_{L^6(\Omega)}
  \le C\big(\|\na S_k^{\alpha_n-1/2}(n_k)\|_{L^2(\Omega)} 
  \label{2.T12} \\
  &+ \|\na S_k^{\alpha_p-1/2}(p_k)\|_{L^2(\Omega)}
  + \|\na S_k^{\alpha_D-1/2}(D_k)\|_{L^2(\Omega)}\big)^{
  \theta(\alpha_0)/(2\alpha_0-1)} + C. \nonumber
\end{align}

Next, we estimate the $L^3(\Omega)$ norm of $\na V_k$. We use the elliptic regularity for the Poisson equation in Hypothesis (H3) to find that
\begin{equation}\label{2.naVk}
  \|\na V_k\|_{L^3(\Omega)} \le C\|n_k-p_k-D_k+A(x)\|_{L^{3/2}(\Omega)}
  + C.
\end{equation}
If $\min\{\alpha_n,\alpha_p,\alpha_D\}\ge 3/2$, the right-hand side is uniformly bounded with respect to $k$ and time, because of the bounds in Lemma \ref{lem.ener}. Thus, let $\alpha_n<3/2$ (similar estimates hold for $\alpha_p<3/2$ or $\alpha_D<3/2$). We conclude from inequality \eqref{2.vS} with $\beta=\alpha_n-1/2<1$, the Gagliardo--Nirenberg inequality, and estimate \eqref{2.estS} that
\begin{align}\label{2.nkL32}
  \|n_k\|_{L^{3/2}(\Omega)} &\le C\|S_k^{\alpha_n-1/2}(n_k)
  \|_{L^{3/(2\alpha_n-1)}(\Omega)}^{1/(\alpha_n-1/2)} + C \\
  &\le C\|\na S_k^{\alpha_n-1/2}(n_k)
  \|_{L^2(\Omega)}^{\widetilde\theta(\alpha_n)/(\alpha_n-1/2)}
  \|S_k^{\alpha_n-1/2}(n_k)\|_{L^{\alpha_n/(\alpha-1/2)}(\Omega)}^{
  (1-\widetilde\theta(\alpha_n))/(\alpha_n-1/2)} + C \nonumber \\
  &\le C\|\na S_k^{\alpha_n-1/2}(n_k)
  \|_{L^2(\Omega)}^{\widetilde\theta(\alpha_n)/(\alpha_n-1/2)} + C,  
  \nonumber
\end{align}
where
$$
  \widetilde\theta(\alpha_n) = \frac{(2\alpha_n-1)(3-2\alpha_n)}{5\alpha_n-3}\in(0,1]
$$
is decreasing for any $\alpha_n\in(1,3/2)$.

Estimating $p_k$ and $D_k$ in a similar way, we find from \eqref{2.naVk} that 
\begin{align*}
  \|\na V_k\|_{L^3(\Omega)} 
  &\le C\sum_{v=n_k,p_k,D_k}
  \|\na S_k^{\alpha_v-1/2}(v_k)\|_{L^2(\Omega)}^{
  \widetilde\theta(\alpha_v)/(\alpha_v-1/2)} + C \\
  &\le C\sum_{v=n_k,p_k,D_k}
  \|\na S_k^{\alpha_v-1/2}(v_k)\|_{L^2(\Omega)}^{
  \widetilde\theta(\alpha_0)/(\alpha_0-1/2)} + C,
\end{align*}
where $\alpha_0=\min\{\alpha_n,\alpha_p,\alpha_D\}$. We combine this estimate and \eqref{2.T12} to infer from \eqref{2.naSk} after integration over $(0,T)$ that 
\begin{align}\label{2.TV}
  \int_0^T&\sum_{v=n_k,p_k,D_k}\|\na S_k^{\alpha_v-1/2}(v)\|_{L^2(\Omega)}^2
  \le C + \int_0^T\sum_{v=n_k,p_k,D_k}\|T_k(n_k)^{1/2}\|_{L^6(\Omega)}^2
  \|\na V_k\|_{L^3(\Omega)}^2\dt \\
  &\le C\int_0^T\sum_{v=n_k,p_k,D_k}
  \|\na S_k^{\alpha_v-1/2}(v)\|_{L^2(\Omega)}^{
  2\theta(\alpha_0)/(2\alpha_0-1)
  +2\widetilde\theta(\alpha_0)/(\alpha_0-1/2)}
  \dt + C. \nonumber
\end{align}
This yields the desired bound if the exponent on the right-hand side is smaller than two or, equivalently, if
$$
  \frac{\theta(\alpha_0)}{2\alpha_0-1} + \frac{\widetilde\theta(\alpha_0)}{
  \alpha_0-1/2} = \frac{9-5\alpha_0}{5\alpha_0-3} < 1,
$$
and this is the case if and only if $\alpha_0>6/5$, finishing the proof.
\end{proof}

We need a uniform bound for the time derivative of the approximate densities.

\begin{lemma}\label{lem.time}
Let $\alpha_n,\alpha_p,\alpha_D>6/5$. Then there exists $C>0$ independent of $k$ such that
\begin{align*}
  \|\pa_t n_k\|_{L^2(0,T;\WD^{1,\beta_n}(\Omega)')}
  + \|\pa_t p_k\|_{L^2(0,T;\WD^{1,\beta_p}(\Omega)')}
  + \|\pa_t D_k\|_{L^2(0,T;W^{1,\beta_D}(\Omega)')} &\le C,
\end{align*}
where $\beta_v=2\alpha_v/(\alpha_v+1)>1$ for $v=n,p,D$.
\end{lemma}

\begin{proof}
Because of
$$
  \|\pa_t n_k\|_{L^2(0,T;\WD^{1,\beta_n}(\Omega)')}
  \le \alpha_n\|T_k(n_k)^{\alpha_n-1}\na n_k
  \|_{L^2(0,T;L^{\beta_n}(\Omega))}
  + \|T_k(n_k)\na V_k\|_{L^2(0,T;L^{\beta_n}(\Omega))},
$$
we only need to estimate the two terms on the right-hand side.
We know from \eqref{2.TV} in the proof of Lemma \ref{lem.naSk} that
\begin{align*}
  \|T_k&(n_k)^{1/2}\na V_k\|_{L^2(\Omega_T)}^2
  \le \int_0^T\|T_k(n_k)^{1/2}\|_{L^6(\Omega)}^2
  \|\na V_k\|_{L^3(\Omega)}^2\dt \\
  &\le C\int_0^T\sum_{v=n_k,p_k,D_k}\|\na S_k(v)^{\alpha_v-1/2}\|_{L^2(\Omega)}^2\dt + C \le C,
\end{align*}
and similarly for the terms involving $T_k(p_k)$ and $T_k(D_k)$. The diffusion term in \eqref{2.nk} is written as
$$
  \alpha_n T_k(n_k)^{\alpha_n-1}\na n_k
  = \na S_k^{\alpha_n}(n_k)
  = \frac{\alpha_n}{\alpha_n-1/2}T_k(n_k)^{1/2}\na S_k^{\alpha_n-1/2}
  (n_k).
$$
Then the bounds for $T_k(n_k)^{1/2}$ in $L^\infty(0,T;L^{2\alpha_n}(\Omega))$ from Lemma \ref{lem.ener} and for $\na S_k^{\alpha_n-1/2}(n_k)$ in $L^2(\Omega_T)$ from Lemma \ref{lem.naSk} imply that the diffusion term is uniformly bounded in $L^2(0,T;L^{\beta_n}(\Omega))$ with $\beta_n=2\alpha_n/(\alpha_n+1)$. Hence, $(\pa_t n_k)$ is bounded in $L^2(0,T;\WD^{1,\beta_n}(\Omega)')$. Again, the proof for $\pa_t p_k$ and $\pa_t D_k$ is similar, noting that the no-flux boundary conditions for $D_k$ yield a slightly different space.
\end{proof}

Next, we prove bounds for the gradients of the approximate densities without cutoff.

\begin{lemma}\label{lem.grad}
There exists $C>0$ independent of $k$ such that for $v=n_k,p_k,D_k$,
\begin{align*}
  \|v\|_{L^2(0,T;W^{1,2\alpha_v/(3-\alpha_v)}(\Omega))}
  \le C &\quad\mbox{if }6/5<\alpha_v\le 3/2, \\
  \|v\|_{L^2(0,T;W^{1,\alpha_n}(\Omega))} \le C
  &\quad\mbox{if }3/2<\alpha_v\le 2.
\end{align*}
Since $2\alpha_v/(3-\alpha_v)>\alpha_v$ for $\alpha_v>1$, we have a uniform bound for $v$ in $L^2(0,T;W^{1,\alpha_v}(\Omega))$ for all $6/5<\alpha_v\le 2$.
\end{lemma}

\begin{proof}
Let first $6/5<\alpha_n\le 3/2$. Then, by the chain rule \eqref{2.chain},
$$
  (\alpha_n-1/2)\na n_k = T_k(n_k)^{3/2-\alpha_n}
  \na S_k^{\alpha_n-1/2}(n_k).
$$
The first factor on the right-hand side is bounded in $L^\infty(0,T;L^{\alpha_n/(3/2-\alpha_n)})$ (see Lemma \ref{lem.ener}), while the second factor is bounded in $L^2(0,T;L^2(\Omega))$ (see Lemma \ref{lem.naSk}). Thus, the product, and consequently $(\na n_k)$, is bounded in $L^2(0,T;L^{2\alpha_n/(3-\alpha_n)}(\Omega))$. It follows from the Poincar\'e inequality that $(n_k)$ is bounded in $L^2(0,T;W^{1,2\alpha_n/(3-\alpha_n)}(\Omega))$.

Second, let $\alpha_n>3/2$. We use $S_k^0(n_k)-S_k^0(\nD)$ as a test function in the weak formulation of \eqref{2.nk}. By the chain rule \eqref{2.chain},
\begin{align}\label{2.aux}
  \int_\Omega\big(&R_k^1(n_k(T))-n_k(T) S_k^0(\nD)\big)\dx
  - \int_\Omega\big(R_k^1(n_k(0))-n_k(0) S_k^0(\nD)\big)\dx \\
  &= \int_0^T\int_\Omega\big(\na S_k^{\alpha_n}(n_k) - T_k(n_k)\na V_k\big)\cdot\na (S_k^0(n_k)-S_k^0(\nD))\dx\dt. \nonumber
\end{align}
Again by the chain rule \eqref{2.chain}, we have
$\na S_k^0(n_k) = \na n_k/T_k(n_k)$ and 
$$
  \na S_k^{\alpha_n}(n_k)\cdot\na S_k^0(n_k)
  = \alpha_n T_k(n_k)^{\alpha_n-2}|\na n_k|^2
  = \frac{4}{\alpha_n}|\na S_k^{\alpha_n/2}(n_k)|^2,
$$
and we conclude from \eqref{2.aux} and similar arguments as in the proof of Lemma \ref{lem.ener} that
$$
   \int_\Omega R_k^1(n_k(T))\dx
   \le C - \frac{4}{\alpha_n}\int_0^T\int_\Omega
   |\na S_k^{\alpha_n/2}(n_k)|^2\dx\dt 
   - \int_0^T\int_\Omega \na V_k\cdot\na n_k\dx\dt.
$$
It remains to estimate the last term. Writing $\na n_k = (2/\alpha_n)
T_k(n_k)^{1-\alpha_n/2}\na S_k^{\alpha_n/2}(n_k)$, we obtain from H\"older's inequality
\begin{align*}
  -\int_0^T&\int_\Omega \na V_k\cdot\na n_k\dx\dt
  \le C\int_0^T\|\na V_k\|_{L^3(\Omega)}
  \|T_k(n_k)^{1-\alpha_n/2}\|_{L^6(\Omega)}
  \|\na S_k^{\alpha_n/2}(n_k)\|_{L^2(\Omega)}\dt \\
  &\le C\|T_k(n_k)^{1-\alpha_n/2}\|_{L^\infty(0,T;L^6(\Omega))}
  \|\na V_k\|_{L^2(0,T;L^3(\Omega))}
  \|\na S_k^{\alpha_n/2}(n_k)\|_{L^2(0,T;L^2(\Omega))}.
\end{align*}
We know that the $L^\infty(0,T;L^{\alpha_n}(\Omega))$ norm of $T_k(n_k)$ is uniformly bounded. Hence, since $\alpha_n>3/2$ implies that $6(1-\alpha_n/2)<\alpha_n$, the $L^\infty(0,T;L^6(\Omega))$ norm of $(T_k(n_k)^{1-\alpha_n/2})$ is bounded too. Here, we need the assumption $\alpha_n\le 2$ to ensure that $1-\alpha_n/2\ge 0$. Furthermore, the $L^2(0,T;L^3(\Omega))$ norm of $\na V_k$ is bounded by the $L^2(0,T;L^{3/2}(\Omega))$ norms of $n_k$, $p_k$, and $D_k$, which are bounded in view of estimate \eqref{2.nkL32}. We infer that
$$
  -\int_0^T\int_\Omega \na V_k\cdot\na n_k\dx\dt
  \le C\|\na S_k^{\alpha_n/2}(n_k)\|_{L^2(0,T;L^2(\Omega))}
$$
and eventually
\begin{align*}
  \int_\Omega R_k^1(n_k(T))\dx
  \le C - \frac{4}{\alpha_n^2}\|\na S_k^{\alpha_n/2}(n_k)\|_{
  L^2(\Omega_T)}^2
  +  C\|\na S_k^{\alpha_n/2}(n_k)\|_{L^2(\Omega_T)}.
\end{align*}
Since the quadratic term dominates the linear one, we obtain a uniform bound for $\na S_k^{\alpha_n/2}(n_k)$ in $L^2(\Omega_T)$. Then
\begin{align*}
  \|\na n_k\|_{L^2(0,T;L^{\alpha_n}(\Omega))}
  &= \frac{2}{\alpha_n}\int_0^T\|T_k(n_k)^{1-\alpha_n/2}\|_{
  L^{2\alpha_n/(2-\alpha_n)}(\Omega)}
  \|\na S_k^{\alpha_n/2}(n_k)\|_{L^2(\Omega)}\dt \\
  &\le \frac{2}{\alpha_n}\|T_k(n_k)\|_{
  L^\infty(0,T;L^{\alpha_n}(\Omega))}^{1-\alpha_n/2}
  \|\na S_k^{\alpha_n/2}(n_k)\|_{L^1(0,T;L^2(\Omega))},
\end{align*}
and we conclude by observing that $T_k(n_k)$ is bounded in $L^\infty(0,T;L^{\alpha_n}(\Omega))$ by Lemma \ref{lem.ener}. The arguments for $p_k$ and $D_k$ are analogous.
\end{proof}


\subsection{Limit $k\to\infty$}

In view of Lemmas \ref{lem.time} and \ref{lem.grad} and the compact embedding $W^{1,\alpha_n}(\Omega)\hookrightarrow L^q(\Omega)$ for $1\le q<3\alpha_n/(3-\alpha_n)\in(2,6]$ for $6/5<\alpha_n\le 2$, we can apply the Aubin--Lions lemma to infer the existence of a subsequence that is not relabeled such that
\begin{align*}
  n_k\to n\quad\mbox{strongly in }L^{q}(\Omega_T)\mbox{ as }
  k\to\infty.
\end{align*}
In particular, $T_k(n_k)\to n$ a.e.\ in $\Omega_T$ and $T_k(n_k)\ge 1/k$ imply that $n\ge 0$ in $\Omega_T$. The choice $q=\alpha_n$ is admissible and leads to
$$
  |n_k|^{\alpha_n-1}\to n^{\alpha_n-1}\quad\mbox{strongly in }
  L^{\alpha_n/(\alpha_n-1)}(\Omega_T).
$$
The uniform bounds in Lemma \ref{lem.time} and \ref{lem.grad} show that, up to subsequences,
\begin{align*}
  \pa_t n_k\rightharpoonup\pa_t n &\quad\mbox{weakly in }
  L^2(0,T;\WD^{1,\beta_n}(\Omega)'), \\
  \na n_k\rightharpoonup\na n &\quad\mbox{weakly in }
  L^2(0,T;L^{\alpha_n}(\Omega)),
\end{align*}
recalling that $\beta_n=2\alpha_n/(\alpha_n+1)$. Moreover, it follows from the $L^\infty(0,T;H^1(\Omega))$ bound on $V_k$ in Lemma \ref{lem.ener} that, up to a subsequence,
$$
  \na V_k\rightharpoonup\na V\quad\mbox{weakly in }L^2(\Omega_T).
$$
These convergences imply that
\begin{align*}
  \alpha_n T_k(n_k)^{\alpha_n-1}\na n_k
  \rightharpoonup \alpha_n n^{\alpha_n-1}\na n = \na n^{\alpha_n}
  &\quad\mbox{weakly in }L^1(\Omega_T), \\
  n_k\na V_k\rightharpoonup n\na V
  &\quad\mbox{weakly in }L^1(\Omega_T).
\end{align*}

We know from the energy estimate in Lemma \ref{lem.ener} that
$(T_k(n_k)^{1/2})$ is bounded in $L^\infty(0,T;$ $L^{2\alpha_n}(\Omega))$ and
$(T_k(n_k)^{1/2}(\alpha_n T_k(n_k)^{\alpha_n-2}\na n_k-\na V_k))$ is bounded in $L^2(\Omega_T)$. Consequently, its product
$$
  \alpha_n T_k(n_k)^{\alpha_n-1}\na n_k - T_k(n_k)\na V_k
$$
is uniformly bounded in $L^2(0,T;L^{2\alpha_n/(\alpha_n+1)}(\Omega))$. We can identify the weak limit with $J_n=\na n^{\alpha_n}-n\na V$, showing that $J_n\in L^2(0,T;L^{2\alpha_n/(\alpha_n+1)}(\Omega))$.
The convergences for $p_k$ and $D_k$ are proved in a similar way.
This shows the existence statement in Theorem \ref{thm.ex}. 

\begin{remark}\label{rem.D}\rm
We may allow for the case $\alpha_n,\alpha_p>1$ and $\alpha_D=1$. Consider first the two-dimensional case. Hypothesis (H3) on elliptic regularity can be replaced by
\begin{equation}\label{2.remV}
  \|\na V\|_{L^3(\Omega)} \le C\|n-p-D+A\|_{L^{6/5}(\Omega)} + C.
\end{equation}
To avoid too many technicalities, we consider the original system and compute formally. The free energy gives a priori bounds for $D\log D$ in $L^\infty(0,T;L^1(\Omega))$ and $L^2(0,T;H^1(\Omega))$ as well as for $\sqrt{D}\na(\log D+V)$ in $L^2(\Omega_T)$. We use $D$ as a test function in \eqref{1.D} to find that
\begin{align}\label{2.remD}
  \frac12\int_\Omega &D(t)^2\dx + \int_0^t\int_\Omega|\na D|^2\dx\dd s
  \le C - \int_0^t\int_\Omega D\na V\cdot\na D\dx\dd s \\
  &\le C + C\int_0^t\|D\|_{L^6(\Omega)}^2\|\na V\|_{L^3(\Omega)}^2\dd s
  + \frac12\int_0^t\|\na D\|_{L^2(\Omega)}^2\dd s. \nonumber
\end{align}
The Gagliardo--Nirenberg inequality of \cite{BHN94} shows that for any $\delta>0$ and $q<\infty$, there exists $C(\delta)>0$ such that
\begin{align}\label{2.GN}
  \|D\|_{L^q(\Omega)} &\le \delta\|D\|_{H^1(\Omega)}^{1-1/q}
  \|D\log D\|_{L^1(\Omega)}^{1/q} + C(\delta)\|D\|_{L^1(\Omega)}
  \le \delta\|\na D\|_{L^2(\Omega)}^{1-1/q} + C(\delta),
\end{align}
where we have used the Poincar\'e--Wirtinger inequality and the uniform bounds for $D$ in $L^\infty(0,T;L^1(\Omega))$ in the last step. Hence, with $q=6/5$,
$$
  \|\na V\|_{L^3(\Omega)} \le C(n,p) + \|D\|_{L^{6/5}(\Omega)}
  \le C(n,p) + C\|\na D\|_{L^2(\Omega)}^{1/6},
$$
where $C(n,p)>0$ is controlled by the estimates in the proof of Theorem \ref{thm.ex}. It follows from \eqref{2.remD} and \eqref{2.GN} with $q=6$ that
\begin{align*}
  \int_\Omega D(t)^2\dx + \frac12\int_0^t\int_\Omega|\na D|^2\dx\dd s
  \le C(n,p,\delta) + \delta\int_0^t\|\na D\|_{L^2(\Omega)}^{5/3}
  \|\na D\|_{L^2(\Omega)}^{1/3}\dd s. 
\end{align*}
The last term can be absorbed by the left-hand side if $\delta<1/2$. 
This gives an a priori estimate for $D$ in $L^\infty(0,T;L^2(\Omega))$ and $L^2(0,T;H^1(\Omega))$. Because of $n,p\in L^2(0,T;L^{3/2}(\Omega))$, condition \eqref{2.remV} yields
\begin{align*}
  \|\na V\|_{L^3(\Omega_T)}
  &\le C(n,p) + \int_0^T\|D\|_{L^{6/5}(\Omega)}\dt \le C(n,p)
  \quad\mbox{and} \\
  \|D\na V\|_{L^2(0,T;L^{3/2}(\Omega))} 
  &\le \|D\|_{L^\infty(0,T;L^2(\Omega))}
  \|\na V\|_{L^3(\Omega_T)} \le C
\end{align*}
and consequently uniform estimates for $\pa_t D=\diver(\na D+D\na V)$ in $L^2(0,T;W^{1,3/2}(\Omega)')$. These bounds are sufficient to conclude compactness via the Aubin--Lions lemma. 

Alternatively, we may use the approach of \cite{JJZ23} to prove the global existence of weak solutions in three space dimensions; however, the regularity becomes in this case $\sqrt{D}\in W^{1,1}(\Omega)$ instead of $D\in H^1(\Omega)$ \cite[Theorem 1.1]{JJZ23}.
\qed\end{remark}

\subsection{Additional regularity}

It remains to prove the additional regularity.

\begin{lemma}\label{lem.alpha-1}
Let $6/5<\alpha_n,\alpha_p,\alpha_D\le 2$. Then
$$
  n^{\alpha_n-1},p^{\alpha_p-1},D^{\alpha_D-1}\in L^2(0,T;H^1(\Omega)).
$$
\end{lemma}

\begin{proof}
We prove the regularity for $D$ only. Using the test function $S_k^{\alpha_D-2}(D_k)$ in the weak formulation of \eqref{2.Dk}, setting $\alpha:=\alpha_D$, and using the chain rule \eqref{2.chain} leads to
\begin{align*}
  \frac{1}{\alpha-1}&\int_\Omega R_k^{\alpha-1}(D_k(t))\dx
  - \frac{1}{\alpha-1}\int_\Omega R_k^{\alpha-1}(D_k(0))\dx \\
  &= \frac{\alpha(2-\alpha)}{(\alpha-1)^2}
  \int_0^t\int_\Omega|\na S_k^{\alpha-1}(D_k)|^2\dx\dd s 
  - \frac{2-\alpha}{\alpha-1}\int_0^t\int_\Omega 
  \na S_k^{\alpha-1}(D_k)\cdot\na V_k\dx\dd s,
\end{align*} 
and similarly if $\alpha=2$, for which $S_k^0(D_k)$ is logarithmic.
We know already that $R_k^{\alpha-1}(D_k)$ is uniformly bounded in $L^\infty(0,T;L^1(\Omega))$. We apply Young's inequality to the second term on the right-hand side, leading eventually to
\begin{align*}
  \int_0^t\int_\Omega|\na S_k^{\alpha-1}(D_k)|^2\dx\dd s
  \le C\int_0^t\int_\Omega|\na V_k|^2\dx\dd s 
  + C\int_\Omega R_k^{\alpha-1}(D_k(t))\dx + C.
\end{align*}
Since the right-hand side is uniformly bounded, we infer that
$(\na S_k^{\alpha-1}(D_k))$ is bounded in $L^2(\Omega_T)$. The pointwise convergence of $(D_k)$ shows that $S_k^{\alpha-1}(D_k)\to D^{\alpha-1}$ pointwise a.e.\ in $\Omega_T$ as $k\to\infty$. Thus, $\na D^{\alpha-1}\in L^2(\Omega_T)$, concluding the proof.
\end{proof}


\section{Regularity of solutions}\label{sec.regul}

In this section, we prove Theorem \ref{thm.regul} in three steps. The proof of Lemma \ref{lem.naSk} shows that the densities are bounded in $L^q(0,T;L^{3/2}(\Omega))$ for $q = (5\alpha-3)/(3-2\alpha)>2$. 
First, we improve this regularity to $n,p,D\in L^\infty(0,T;L^{3/2}(\Omega))$. Then we show that $n,p,D\in L^\infty(0,T;L^{q}(\Omega))$ for any $q<\infty$ with a bound depending on $q$. Finally, we prove the final goal $n,p,D\in L^\infty(0,T;L^\infty(\Omega))$.

\begin{lemma}\label{lem.L32}
Let $\alpha_n,\alpha_p,\alpha_D>\alpha^*:=(11+\sqrt{37})/14$. Then
$$
  n,p,D\in L^\infty(0,T;L^{3/2}(\Omega)).
$$
\end{lemma}

\begin{proof}
Let $(n_k,p_k,D_k)$ be a weak solution to \eqref{2.nk}--\eqref{2.JD}. We focus on the estimation of $D_k$ to avoid the boundary data, but the computations for $n_k$ and $p_k$ are similar. The statement follows from the energy estimate in Lemma \ref{lem.ener} if $\alpha_D\ge 3/2$. Therefore, let $\alpha_D<3/2$. In the following, we make an iterative argument.

Set $\gamma_0:=\alpha_D-1>0$. Then, by Lemma \ref{lem.ener}, 
$\|R_k^{\gamma_0+1}(D_k)\|_{L^\infty(0,T;L^{1}(\Omega))}\le C_0$, where $C_0>0$ only depends on the initial data. Assume that 
\begin{equation}\label{3.Cm}
  \|R_k^{\gamma_m+1}(D_k)\|_{L^\infty(0,T;L^{1}(\Omega))}\le C_m
\end{equation} 
for some $C_m>0$ and $\gamma_m>\alpha_D-1$. We wish to prove that there exist $C_{m+1}>0$ and $\gamma_{m+1}>\gamma_m$ such that 
$$
  \|R_k^{\gamma_{m+1}+1}(D_k)\|_{L^\infty(0,T;L^{1}(\Omega))}\le C_{m+1}.
$$

To simplify the notation, we set $\alpha:=\alpha_D$ and $\gamma:=\gamma_{m+1}$. We use the test function $S_k^\gamma(D_k)$ in the weak formulation of \eqref{2.Dk} and apply the chain rule \eqref{2.chain}:
\begin{align*}
  \frac{1}{\gamma+1}&\frac{\dd}{\dt}\int_\Omega R_k^{\gamma+1}(D_k)\dx
  = \langle\pa_t D_k,S_k^\gamma(D_k)\rangle \\
  &= -\int_\Omega(\na S_k^{\alpha}-T_k(D_k)\na V_k)
  \cdot\na S_k^\gamma(D_k)\dx
  = -\frac{4\alpha\gamma}{(\alpha+\gamma)^2}\int_\Omega
  |\na S_k^{(\alpha+\gamma)/2}(D_k)|^2\dx \\
  &\phantom{xx}+ \frac{2\gamma}{\alpha+\gamma}
  \int_\Omega T_k(D_k)^{(\gamma-\alpha+2)/2}\na V_k\cdot
  \na S_k^{(\alpha+\gamma)/2}(D_k)\dx.
\end{align*}
It follows from H\"older's inequality that
\begin{align}\label{3.Rg}
  \frac{1}{\gamma+1}&\frac{\dd}{\dt}\int_\Omega R_k^{\gamma+1}(D_k)\dx
  + \frac{4\alpha\gamma}{(\alpha+\gamma)^2}\int_\Omega
  |\na S_k^{(\alpha+\gamma)/2}(D_k)|^2\dx \\
  &\le C\|T_k(D_k)^{(\gamma-\alpha+2)/2}\|_{L^6(\Omega)}
  \|\na V_k\|_{L^3(\Omega)}
  \|\na S_k^{(\alpha+\gamma)/2}(D_k)\|_{L^2(\Omega)}. \nonumber
\end{align}
The first factor in the right-hand side is estimated by means of the Gagliardo--Nirenberg inequality and estimate \eqref{2.TS} to switch from $T_k$ to $S_k$:
\begin{align*}
  \|T_k&(D_k)^{(\gamma-\alpha+2)/2}\|_{L^6(\Omega)}
  = \|T_k(D_k)^{(\alpha+\gamma)/2}\|_{L^{6(\gamma-\alpha+2)
  /(\alpha+\gamma)}(\Omega)}^{(\gamma-\alpha+2)/(\alpha+\gamma)} \\
  &\le C\|S_k^{(\alpha+\gamma)/2}(D_k)\|_{L^{6(\gamma-\alpha+2)
  /(\alpha+\gamma)}(\Omega)}^{(\gamma-\alpha+2)/(\alpha+\gamma)} + C \\
  &\le C\big(\|\na S_k^{(\alpha+\gamma)/2}(D_k)\|_{L^2(\Omega)}^\theta
  \|S_k^{(\alpha+\gamma)/2}(D_k)\|_{
  L^{2(\gamma_m+1)/(\alpha+\gamma)}(\Omega)}^{1-\theta} \big)^{(\gamma-\alpha+2)/(\alpha+\gamma)} + C,
\end{align*}
where
$$
  \theta = \frac{(\alpha+\gamma)(5-3\alpha+3\gamma-\gamma_m)}{
  (3\alpha+3\gamma-\gamma_m-1)(\gamma-\alpha+2)}\in[0,1]
$$
holds since $\alpha\ge 1$. We deduce from relation \eqref{2.SR} between $S_k^{(\alpha+\gamma)/2}$ and $R_k^{\gamma_m+1}$ and the recursion assumption \eqref{3.Cm} that
\begin{align*}
  \|S_k^{(\alpha+\gamma)/2}(D_k)\|_{L^{2(\gamma_m+1)/(\alpha+\gamma)}
  (\Omega)}^{(\gamma-\alpha+2)/(\alpha+\gamma)}
  &\le C\|R_k^{\gamma_m+1}(D_k)\|_{L^{1}(\Omega)}^{
  (\gamma-\alpha+2)/(2(\gamma_m+1))} + C \\
  &\le CC_m^{(\gamma-\alpha+2)/(2(\gamma_m+1))}+C =: K_m. \nonumber
\end{align*}
This yields (choosing $K_m\ge 1$ so that $K_m^{1-\theta}\le K_m$)
\begin{align}\label{3.drift1}
  \|T_k(D_k)^{(\gamma-\alpha+2)/2}\|_{L^6(\Omega)}
  \le C\big(K_m\|\na S_k^{(\alpha+\gamma)/2}(D_k)\|_{L^2(\Omega)}^{
  (5-3\alpha+3\gamma-\gamma_m)/(3\alpha+3\gamma-\gamma_m-1)} 
  + 1\big).
\end{align}

To estimate the second factor in \eqref{3.Rg}, the norm of $\na V_k$, we first apply the Gagliardo--Nirenberg inequality:
\begin{align*}
  \int_0^T\|D_k\|_{L^{3/2}(\Omega)}^{(5\alpha-3)/(3-2\alpha)}\dt
  &\le C\int_0^T\|D_k\|_{W^{1,2\alpha/(3-\alpha)}(\Omega)
  }^{\widetilde\theta(5\alpha-3)/(3-2\alpha)}\|D_k\|_{
  L^\alpha(\Omega)}^{(1-\widetilde\theta)(5\alpha-3)/(3-2\alpha)}\dt \\
  &\le \|D_k\|_{L^\infty(0,T;L^\alpha(\Omega))}^{(1-\widetilde\theta)
  (5\alpha-3)/(3-2\alpha)}
  \int_0^T\|D_k\|_{W^{1,2\alpha/(3-\alpha)}(\Omega)
    }^{\widetilde\theta(5\alpha-3)/(3-2\alpha)}\dt,
\end{align*}
where 
$$
  \widetilde\theta = \frac{6-4\alpha}{5\alpha-3}\in(0,1)
$$
holds since we assumed $1<\alpha<3/2$. By Lemma \ref{lem.grad}, the integral on the right-hand side is uniformly bounded since
$\widetilde\theta(5\alpha-3)/(3-2\alpha)=2$. We observe that we obtain in a similar way uniform bounds for $n_k$ and $p_k$ in the space $L^{(5\alpha-3)/(3-2\alpha)}(0,T;L^{3/2}(\Omega))$. Then we deduce from Hypothesis (H3) that 
\begin{equation}\label{3.drift2}
  \int_0^T\|\na V_k\|_{L^3(\Omega)}^{(5\alpha-3)/(3-2\alpha)}\dt\le C,
\end{equation}
which improves the $L^2(0,T;L^3(\Omega))$ bound for $\na V_k$ proved before, since $(5\alpha-3)/(3-2\alpha)\in(5,\infty)$ for $6/5<\alpha<3/2$.

Now, inserting estimates \eqref{3.drift1} and \eqref{3.drift2} into \eqref{3.Rg}, integrated over time, we infer that
\begin{align*}
  &\frac{1}{\gamma+1}\int_\Omega R_k^{\gamma+1}(D_k(t))\dx
  + \frac{4\alpha\gamma}{(\alpha+\gamma)^2}\int_0^t
  \|\na S_k^{(\alpha+\gamma)/2}(D_k)\|_{L^2(\Omega)}^2\dd s \\
  &\phantom{x}\le \frac{1}{\gamma+1}\int_\Omega R_k^\gamma(D_k(0))\dx
  + C\int_0^t\|\na S_k^{(\alpha+\gamma)/2}(D_k)\|_{L^2(\Omega)}
  \|\na V_k\|_{L^3(\Omega)} \nonumber \\
  &\phantom{xxx}\times\big(K_m
  \|\na S_k^{(\alpha+\gamma)/2}(D_k)\|_{L^2(\Omega)}^{
  (5-3\alpha+3\gamma-\gamma_m)/(3\alpha+3\gamma-\gamma_m-1)} 
  + 1\big)\dd s \nonumber \\
  &\phantom{x}\le C + CK_m\int_0^t\big(
  \|\na S_k^{(\alpha+\gamma)/2}(D_k)\|_{L^2(\Omega)}^r + 1\big)
  \|\na V_k\|_{L^3(\Omega)}\dd s, \nonumber 
\end{align*}
setting
$$
  r = 1 + \frac{5-3\alpha+3\gamma-\gamma_m}{3\alpha+3\gamma-\gamma_m-1}
  = \frac{4+6\gamma-2\gamma_m}{3\alpha+3\gamma-\gamma_m-1}>1.
$$
We apply H\"older's inequality with $q=(5\alpha-3)/(3-2\alpha)$ and then Young's inequality,
\begin{align*}
  &\frac{1}{\gamma+1}\int_\Omega R_k^{\gamma+1}(D_k(t))\dx
  + \frac{4\alpha\gamma}{(\alpha+\gamma)^2}\int_0^t
  \|\na S_k^{(\alpha+\gamma)/2}(D_k)\|_{L^2(\Omega)}^2\dd s \\
  &\phantom{x}\le C + CK_m\bigg(\int_0^t
  \|\na V_k\|_{L^3(\Omega)}^q\dd s\bigg)^{1/q}\bigg(\int_0^t
  \big(\|\na S_k^{(\alpha+\gamma)/2}(D_k)\|_{L^2(\Omega)}^r+1
  \big)^{q/(q-1)}\dd s\bigg)^{(q-1)/q} \\
  &\phantom{x}\le C + \frac{(\alpha+\gamma)^2}{\alpha\gamma}CK_m^2
  \int_0^t\|\na V_k\|_{L^3(\Omega)}^q\dd s
  + \frac{2\alpha\gamma}{(\alpha+\gamma)^2}\int_0^t
  \|\na S_k^{(\alpha+\gamma)/2}(D_k)\|_{L^2(\Omega)}^{rq/(q-1)}\dd s.
\end{align*}
By \eqref{3.drift2}, the second term on the right-hand side is bounded, while the last term can be absorbed by the left-hand side if $rq/(q-1) = 2$, which is equivalent to
$$
  \gamma_{m+1} = \gamma 
  = \frac{21\alpha^2 - 35\alpha + 12}{9-6\alpha} + \frac{\gamma_m}{3}.
$$
Our requirement $\gamma_{m+1}>\gamma_m$ is equivalent to 
$$
  \gamma_m < \frac{21\alpha^2-35\alpha+12}{6-4\alpha}.
$$
In particular, $\gamma_m>\alpha-1$ has to be satisfied (since $\gamma_0=\alpha-1$). This leads to the necessary condition 
$$
  \alpha-1 < \frac{21\alpha^2-35\alpha+12}{6-4\alpha},
$$
which is equivalent to $6/5<\alpha<3/2$. We define the recursive sequence of exponents for $6/5<\alpha<3/2$ by
$$
  \gamma_{m+1} = \frac{21\alpha^2-35\alpha+12}{9-6\alpha}
  + \frac{\gamma_m}{3}, \quad m\in\N, \quad \gamma_0 = \alpha-1.
$$
This recursion has the explicit solution 
$$
  \gamma_m = \frac{21\alpha^2-35\alpha+12}{6-4\alpha}
  \bigg(1-\frac{1}{3^m}\bigg) + \frac{\alpha-1}{3^m}, \quad m\in\N.
$$

We infer that
\begin{align*}
  \frac{1}{\gamma+1}\int_\Omega R_k^{\gamma+1}(D_k(t))\dx
  + \frac{2\alpha\gamma}{(\alpha+\gamma)^2}\int_0^t
  \|\na S_k^{(\alpha+\gamma)/2}(D_k)\|_{L^2(\Omega)}^2\dd s
  \le C (1 + \gamma K_m^2).
\end{align*}
Consequently,
$$
  \|R_k^{\gamma+1}(D_k(t))\|_{L^\infty(0,T;L^1(\Omega))}
  \le (\gamma+1)C(1+\gamma K_m^2) =: C_{m+1}.
$$
which is the desired bound. We deduce from Lemma \ref{lem.RST} the uniform bound
$$
  \|T_k(D_k)\|_{L^\infty(0,T;L^{\gamma_{m+1}+1}(\Omega))}^{
  \gamma_{m+1}+1}
  \le C\|R_k^{\gamma+1}(D_k)\|_{L^\infty(0,T;L^1(\Omega))}+C
  \le C(C_{m+1}+1),
$$
and the limit $k\to\infty$ leads to a uniform bound for $D$ in
$L^\infty(0,T;L^{\gamma_{m+1}+1}(\Omega))$ for all
$$
  \gamma_{m+1} < \frac{21\alpha^2-35\alpha+12}{6-4\alpha}.
$$
We wish to reach $\gamma_{m+1}+1=3/2$. Hence, we have to guarantee that
$$
  \frac12 > \frac{21\alpha^2-35\alpha+12}{6-4\alpha},
$$
which yields the restriction $\alpha>(11+\sqrt{37})/14=\alpha^*$. This finishes the proof.
\end{proof}

The bound of Lemma \ref{lem.L32} implies, by Hypothesis (H3) with $r=3$, that $(\na V_k)$ is bounded in $L^\infty(0,T;W^{1,3}(\Omega))$. This helps us to improve the regularity of the densities.

\begin{lemma}
Let $\alpha_n,\alpha_p,\alpha_D>\alpha^*$. Then
$$
  n,p,D\in L^\infty(0,T;L^{q}(\Omega))\quad\mbox{for all }1\le q<\infty.
$$
\end{lemma}

\begin{proof}
We set $\alpha:=\alpha_D$ and choose an arbitrary $\gamma>0$. Computing as in the previous proof, inequality \eqref{3.Rg} (integrated over time) holds in this situation:
\begin{align}\label{3.inter3}
  \frac{1}{\gamma+1}&\int_\Omega R_k^{\gamma+1}(D_k(t))\dx
  + \frac{4\alpha\gamma}{(\alpha+\gamma)^2}\int_0^t
  \|\na S_k^{(\alpha+\gamma)/2}(D_k)\|_{L^2(\Omega)}^2\dd s \\
  &\le \frac{1}{\gamma+1}\int_\Omega R_k^{\gamma+1}(D_k(0))
  \dx \nonumber \\
  &\phantom{xx}
  + C\int_0^t\|T_k(D_k)^{(\gamma-\alpha+2)/2}\|_{L^6(\Omega)}
  \|\na V_k\|_{L^3(\Omega)}
  \|\na S_k^{(\alpha+\gamma)/2}(D_k)\|_{L^2(\Omega)}\dd s. \nonumber
\end{align}
We estimate the norm of $T_k(D_k)^{(\gamma-\alpha+2)/2}$ by using the Gagliardo--Nirenberg inequality, similarly as in the previous proof:
\begin{align*}
  \|T_k&(D_k)^{(\gamma-\alpha+2)/2}\|_{L^6(\Omega)}
  \le C\|S_k^{(\alpha+\gamma)/2}\|_{
  L^{6(\gamma-\alpha+2)/(\alpha+\gamma)}(\Omega)}^{
  (\gamma-\alpha+2)(\alpha+\gamma)} + C \\
  &\le C\big(\|\na S_k^{(\alpha+\gamma)/2}(D_k)\|_{L^2(\Omega)}^\eta
  \|S_k^{(\alpha+\gamma)/2}(D_k)\|_{L^{2\alpha/(\alpha+\gamma)}
  (\Omega)}^{1-\eta}\big)^{(\gamma-\alpha+2)/(\alpha+\gamma)}  + C,
\end{align*}
where
$$
  \eta = \frac{(\alpha+\gamma)(3\gamma-4\alpha+6)}{(\gamma-\alpha+2)
  (2\alpha+3\gamma)}\in(0,1).
$$
By the energy estimates of Lemma \ref{lem.ener} and relation \eqref{2.SR}, the $L^\infty(0,T;L^{2\alpha/(\alpha+\gamma)}(\Omega))$ norm of $S_k^{(\alpha+\gamma)/2}(D_k)$ is uniformly bounded. Hence, 
$$
  \|T_k(D_k)^{(\gamma-\alpha+2)/2}\|_{L^6(\Omega)}
  \le C\|\na S_k^{(\alpha+\gamma)/2}(D_k)\|_{L^2(\Omega)
  }^{\eta(\gamma-\alpha+2)/(\alpha+\gamma)} + C.
$$
We insert this estimate into \eqref{3.inter3} and take into account that the $L^\infty(0,T;L^3(\Omega))$ norm of $\na V_k$ is uniformly bounded:
\begin{align*}
  \frac{1}{\gamma+1}&\int_\Omega R_k^{\gamma+1}(D_k(t))\dx
  + \frac{4\alpha\gamma}{(\alpha+\gamma)^2}\int_0^t
  \|\na S_k^{(\alpha+\gamma)/2}(D_k)\|_{L^2(\Omega)}^2\dd s \\
  &\le C + C\int_0^t\|\na S_k^{(\alpha+\gamma)/2}(D_k)\|_{L^2(\Omega)}^s
  \dd s,
\end{align*}
where
$$
  s = 1 + \eta\frac{\gamma-\alpha+2}{\alpha+\gamma} 
  = 2\frac{3\gamma-\alpha+3}{2\alpha+3\gamma} < 2
$$
holds because of $\alpha>1$. Therefore, we can apply the Young inequality $ab\le \eps a^c + \eps^{-1/(c-1)}b^{c/(c-1)}$ for $a,b\ge 0$, $\eps>0$, $c>1$ with the choice $b=C$, $\eps= 2\alpha\gamma/(\alpha+\gamma)^2$, and $c=2/s>1$ to find that
\begin{align*}
  \frac{1}{\gamma+1}&\int_\Omega R_k^{\gamma+1}(D_k(t))\dx
  + \frac{4\alpha\gamma}{(\alpha+\gamma)^2}\int_0^t
  \|\na S_k^{(\alpha+\gamma)/2}(D_k)\|_{L^2(\Omega)}^2\dd s \\
  &\le C + \frac{2\alpha\gamma}{(\alpha+\gamma)^2}\int_0^t
  \|\na S_k^{(\alpha+\gamma)/2}(D_k)\|_{L^2(\Omega)}^2\dd s
  + C^{2/(2-s)}
  \bigg(\frac{2\alpha\gamma}{(\alpha+\gamma)^2}\bigg)^{-s/(2-s)}.
\end{align*}
The second term on the right-hand side can be absorbed by the left-hand side. Then, writing
$(2\alpha\gamma/(\alpha+\gamma)^2)^{-1}\le C(\gamma+1)$, where $C>0$ depends on $\alpha$,
\begin{align*}
  \int_\Omega R_k^{\gamma+1}(D_k(t))\dx
  \le C(\gamma+1) + C^{(2\alpha+3\gamma)/(3\alpha-3)}
  C(\gamma+1)^{1+(3\gamma-\alpha+3)/(3\alpha-3)}.
\end{align*}
Thus, we deduce from Lemma \ref{lem.RST} the estimate
\begin{align}\label{3.gamma}
  \|T_k(D_k)\|_{L^\infty(0,T;L^{\gamma+1}(\Omega))}
  &\le C\|R_k^{\gamma+1}(D_k)\|_{L^\infty(0,T;L^1(\Omega))
  }^{1/(\gamma+1)} + C \\
  &\le C(\gamma+1)^{1/(\gamma+1)} 
  + C(\gamma+1)^{(2\alpha+3\gamma)/((3\alpha-3)(\gamma+1))}. \nonumber
\end{align}
This provides a uniform bound for $D_k$ and, after the limit $k\to\infty$, for $D$ in $L^\infty(0,T;L^{\gamma+1}(\Omega))$ for any
$\gamma<\infty$. Unfortunately, the right-hand side of \eqref{3.gamma} diverges as $\gamma\to\infty$, and we cannot conclude a uniform bound in $L^{\infty}(\Omega)$. 
\end{proof}

Finally, we prove the last statement of Theorem \ref{thm.regul}. 
\begin{lemma}
Let $\alpha_n,\alpha_p,\alpha_D>\alpha^*$ and let Hypothesis (H3) hold with $r>3$. Then $n,p,D\in L^\infty(0,T;L^\infty(\Omega))$.
\end{lemma}

\begin{proof}
The idea of the proof is to perform an Alikakos-type iteration \cite{Ali79}. We know from Theorem \ref{thm.regul} that $n,p,D\in L^\infty(0,T;L^q(\Omega))$ for any $q<\infty$. By Hypothesis (H3) with $r=3+\eta>3$, this implies the bound
\begin{equation}\label{3.Veta}
  \|V\|_{L^\infty(0,T;W^{1,3+\eta}(\Omega))} \le C + C\|n-p-D+A\|_{L^\infty(0,T;L^q(\Omega))}
	\le C
\end{equation}
for $q=(9+3\eta)/(6+\eta)>3/2$. We shall only show the case of pure Neumann boundary conditions to simplify the presentation. Then, setting $\alpha:=\alpha_D$ and using $D^\gamma$ for some $\gamma>0$ as a test function in the weak formulation of \eqref{1.D}, we find that
\begin{align}\label{3.Dgamma}
  \frac{1}{\gamma+1}&\frac{\dd}{\dt}\|D\|_{L^{\gamma+1}(\Omega)}^{\gamma+1}
	+ \frac{4\alpha\gamma}{(\alpha+\gamma)^2}\|\na D^{(\alpha+\gamma)/2}\|_{L^2(\Omega)}^2 \\
	&= \frac{2\gamma}{\alpha+\gamma}\int_\Omega D^{(\gamma-\alpha+2)/2}\na V\cdot
	\na D^{(\alpha+\gamma)2}\dx \nonumber \\
	&\le \frac{2\gamma}{\alpha+\gamma}\|D^{(\gamma-\alpha+2)/2}\|_{L^{6-\mu}(\Omega)}
	\|\na V\|_{L^{3+\eta}(\Omega)}\|\na D^{(\alpha+\gamma)/2}\|_{L^2(\Omega)} \nonumber \\
	&\le C\|D^{(\gamma-\alpha+2)/2}\|_{L^{6-\mu}(\Omega)}
	\|\na D^{(\alpha+\gamma)/2}\|_{L^2(\Omega)}, \nonumber 
\end{align}
where $\mu=4\eta/(1+\eta)$ is determined from H\"older's inequality via $1/(6-\mu)+1/(3+\eta)+1/2=1$, and we have used the bound \eqref{3.Veta} which is uniform in $\gamma$ and $t>0$. We estimate the first factor on the right-hand side, using $(\gamma-\alpha+2)/(\alpha+\gamma)<1$:
\begin{align*}
  \|D^{(\gamma-\alpha+2)/2}\|_{L^{6-\mu}(\Omega)}
	&= \|D^{(\alpha+\gamma)/2}\|_{L^{(6-\mu)(\gamma-\alpha+2)
	/(\alpha+\gamma)}(\Omega)}^{(\gamma-\alpha+2)/(\alpha+\gamma)} \\
	&\le C(\Omega)\|D^{(\alpha+\gamma)/2}\|_{L^{6-\mu}(\Omega)}^{(\gamma-\alpha+2)
	/(\alpha+\gamma)}
	\le C\big(1+\|D^{(\alpha+\gamma)/2}\|_{L^{6-\mu}(\Omega)}\big).
\end{align*}
We deduce from the Gagliardo--Nirenberg inequality with $\theta=(30-6\mu)/(30-5\mu)\in(0,1)$ that
$$
  \|D^{(\alpha+\gamma)/2}\|_{L^{6-\mu}(\Omega)}
	\le C\|\na D^{(\alpha+\gamma)/2}\|_{L^2(\Omega)}^\theta
	\|D^{(\alpha+\gamma)/2}\|_{L^1(\Omega}^{1-\theta} + C\|D^{(\alpha+\gamma)/2}\|_{L^1(\Omega)}.
$$
Thus, the right-hand side of \eqref{3.Dgamma} can be bounded as
\begin{align}\label{3.Dgamma2}  
  \|D&^{(\gamma-\alpha+2)/2}\|_{L^{6-\mu}(\Omega)}
	\|\na D^{(\alpha+\gamma)/2}\|_{L^2(\Omega)} \\
	&\le C\big(1 + \|\na D^{(\alpha+\gamma)/2}\|_{L^2(\Omega)}^\theta
	\|D^{(\alpha+\gamma)/2}\|_{L^1(\Omega}^{1-\theta} 
	+ \|D^{(\alpha+\gamma)/2}\|_{L^1(\Omega)}\big) 
	\|\na D^{(\alpha+\gamma)/2}\|_{L^2(\Omega)} \nonumber \\
	&\le C\|\na D^{(\alpha+\gamma)/2}\|_{L^2(\Omega)}^{1+\theta}
	\big(1 + \|D^{(\alpha+\gamma)/2}\|_{L^1(\Omega}^{1-\theta}\big) \nonumber \\
	&\phantom{xx}+ C\|\na D^{(\alpha+\gamma)/2}\|_{L^2(\Omega)}
	\|D^{(\alpha+\gamma)/2}\|_{L^1(\Omega)}. \nonumber 
\end{align}

Next, we apply Young's inequality $ab\le \delta a^p + b^q/(qp^{q/p}\delta^{q/p})$ with
$$
  p = \frac{2}{1+\theta}, \quad q = \frac{2}{1-\theta}, \quad 
	\delta = \frac{2\alpha\gamma}{(\alpha+\gamma)^2}
$$
to the first term and Young's inequality $ab\le \delta a^2 + b^2/(4\delta)$ to the second term on the right-hand side of \eqref{3.Dgamma2}. Observe that $p$ and $q$ depend only on $\mu$ (and hence on $\eta$) but not on $\gamma$. At this point, we need the better regularity of $V$ in $W^{1,3+\eta}(\Omega)$ instead in $W^{1,3}(\Omega)$, since $\eta=0$ would imply that $\theta=1$ and then the norm of $\na D$ in \eqref{3.Dgamma2} is squared and cannot generally be absorbed. We need $1+\theta<2$ which requires that $\eta>0$. Then \eqref{3.Dgamma} becomes
\begin{align*}
  \frac{1}{\gamma+1}&\frac{\dd}{\dt}\|D\|_{L^{\gamma+1}(\Omega)}^{\gamma+1}
	+ \frac{4\alpha\gamma}{(\alpha+\gamma)^2}\|\na D^{(\alpha+\gamma)/2}\|_{L^2(\Omega)}^2 \\
	&\le \frac{2\alpha\gamma}{(\alpha+\gamma)^2}\|\na D^{(\alpha+\gamma)/2}\|_{L^2(\Omega)}^2
	+ \frac{C}{qp^{q/p}}\bigg(\frac{(\alpha+\gamma)^2}{2\alpha\gamma}
	\bigg)^{q/p}
	\big(1 + \|D^{(\alpha+\gamma)/2}\|_{L^1(\Omega}^{1-\theta}\big)^q \\
	&\phantom{xx}+ \frac{2\alpha\gamma}{(\alpha+\gamma)^2}
	\|\na D^{(\alpha+\gamma)/2}\|_{L^2(\Omega)}^2
	 + C\frac{(\alpha+\gamma)^2}{2\alpha\gamma}
	\|D^{(\alpha+\gamma)/2}\|_{L^1(\Omega)}^{2}.
\end{align*}
The first and third term on the right-hand side are absorbed by the left-hand side. Then
\begin{align*}
  \frac{1}{\gamma+1}\frac{\dd}{\dt}\|D(t)\|_{L^{\gamma+1}(\Omega)}^{\gamma+1}
	&\le \frac{C}{qp^{q/p}}\bigg(\frac{(\alpha+\gamma)^2}{2\alpha\gamma}\bigg)^{q/p}
	\big(1 + \|D\|_{L^{(\alpha+\gamma)/2}(\Omega)}^{(\alpha+\gamma)(1-\theta)/2}\big)^q \\
	&\phantom{xx}+ C\bigg(\frac{(\alpha+\gamma)^2}{2\alpha\gamma}\bigg)^2
	\|D\|_{L^{(\alpha+\gamma)/2}(\Omega)}^{\alpha+\gamma}.
\end{align*}
We deduce from $q/p=(1+\theta)/(1-\theta)=(60-11\mu)/\mu$ that there exists $C(\alpha)>0$ such that for $\gamma\ge 1$,
$$
  \bigg(\frac{(\alpha+\gamma)^2}{2\alpha\gamma}\bigg)^{q/p}
	\le (C(\alpha)\gamma)^{q/p} \le C(\alpha,\mu)\gamma^{(60-11\mu)/\mu}, \quad
	\bigg(\frac{(\alpha+\gamma)^2}{2\alpha\gamma}\bigg)^2 \le C(\alpha)\gamma^2.
$$
Moreover, we have $(1-\theta)q/2=1$. This shows that
\begin{align*}
  \frac{1}{\gamma+1}\frac{\dd}{\dt}\|D\|_{L^{\gamma+1}(\Omega)}^{\gamma+1}
	\le C(\alpha,\mu)(\gamma^{(60-11\mu)/\mu} + \gamma^2)
	\big(1 + \|D\|_{L^{(\alpha+\gamma)/2}(\Omega)}^{\alpha+\gamma}\big).
\end{align*}
Consequently, since $(60-11\mu)/\mu>2$ (which follows from $\mu=4\eta/(1+\eta)<4$), integrating the previous inequality over $(0,t)$ and multiplying it by $\gamma+1$, we obtain
\begin{align*}
  \|D(t)&\|_{L^{\gamma+1}(\Omega)}^{\gamma+1}
	\le \|D_I\|_{L^{\gamma+1}(\Omega)}^{\gamma+1} + C(\alpha,\mu)\gamma^\beta\int_0^t
	\big(1 + \|D\|_{L^{(\alpha+\gamma)/2}(\Omega)}^{\alpha+\gamma}\big)\dd s \\
	&\le C(\Omega)\|D_I\|_{L^\infty(\Omega)}^{\gamma+1}
	+ C(T)\gamma^\beta\big(1+\|D\|_{L^\infty(0,T;L^{(\alpha+\gamma)/2}(\Omega))
	}^{\alpha+\gamma}\big),
\end{align*}
where $\beta = (60-11\mu)/\mu+1=10(6-\mu)/\mu\in(0,\infty)$. We take the supremum over $t\in(0,T)$:
\begin{equation}\label{3.Dgamma3}
  \|D\|_{L^\infty(0,T;L^{\gamma+1}(\Omega))}^{\gamma+1}
	\le C\|D_I\|_{L^\infty(\Omega)}^{\gamma+1} + C\gamma^\beta
	\big(1+\|D\|_{L^\infty(0,T;L^{(\alpha+\gamma)/2}(\Omega))}^{\alpha+\gamma}\big).
\end{equation}

The original Alikakos method is based on halving the exponents (which happens if $\alpha=1$), but since we have $\alpha>1$, the argument is slightly different. We set $\gamma_k:=\gamma+1$ and $\gamma_{k-1}:=(\gamma+\alpha)/2$. This gives the recursion
$\gamma_{k-1} = (\gamma_k-1+\alpha)/2$, which can be solved explicitly:
\begin{equation}\label{3.gammak}
  \gamma_k = 2^k(\gamma_0 + 1 - \alpha) + \alpha-1, \quad k\in\N.
\end{equation}
Setting
$$
  b_k := \|D\|_{L^\infty(0,T;L^{\gamma_k}(\Omega))}^{\gamma_k} 
	+ \|D_I\|_{L^\infty(\Omega)}^{\gamma_k} + 1,
$$
we can write \eqref{3.Dgamma3} as
\begin{align*}
  b_k &\le (C+1)\|D_I\|_{L^\infty(\Omega)}^{\gamma_k}
	+ C(\gamma_k-1)^\beta\big(1+\|D\|_{L^\infty(0,T;L^{\gamma_{k-1}}(\Omega))}^{2\gamma_{k-1}}
	\big) + 1 \\
	&\le C\gamma_k^\beta\big(\|D_I\|_{L^\infty(\Omega)}^{2\gamma_{k-1}}
	+ \|D\|_{L^\infty(0,T;L^{\gamma_{k-1}}(\Omega))}^{2\gamma_{k-1}} + 1\big)
	\le C\gamma_k^\beta b_{k-1}^2 \le C^k\gamma_k^\beta b_{k-1}^2,
\end{align*}
using $\gamma_k<2\gamma_{k-1}$. Since $\gamma_k\le 3^{\beta k}$ for sufficiently large $k$, the recursion inequality becomes
$$
  b_k \le C3^{k\beta}b_{k-1}^2 = M^k b_{k-1}^2, \quad\mbox{where }M:=3^\beta C.
$$
We solve this recursion by introducing $c_k:=M^{k+2}b_k$:
$$
  c_k \le M^{2(k+1)}b_{k-1}^2 = (M^{k+1}b_{k-1})^2 = c_{k-1}^2,
$$
which gives $c_k\le c_0^{2^k}$ and consequently,
\begin{align*}
  b_k = M^{-(k+2)}c_k \le M^{-(k+2)}c_0^{2^k} = M^{-(k+2)}(M^2b_0)^{2^k}
	= M^{2^{k+1}-(k+2)}b_0^{2^k}
\end{align*}
We conclude that
\begin{align*}
  \|D\|_{L^\infty(0,T;L^{\gamma_k}(\Omega))}^{\gamma_k} \le b_k
	\le M^{2^{k+1}-(k+2)}\big(\|D_I\|_{L^\infty(\Omega)}^{\gamma_{0}}
	+ \|D\|_{L^\infty(0,T;L^{\gamma_{0}}(\Omega))}^{\gamma_0} + 1\big)^{2^k}
\end{align*}
and, taking the $\gamma_k$-th root,
\begin{equation}\label{3.Dgammak}
  \|D\|_{L^\infty(0,T;L^{\gamma_k}(\Omega))}
	\le M^{(2^{k+1}-(k+2))/\gamma_k}\big(\|D_I\|_{L^\infty(\Omega)
	}^{\gamma_{0}}
	+ \|D\|_{L^\infty(0,T;L^{\gamma_{0}}(\Omega))}^{\gamma_{0}} 
	+ 1\big)^{2^k/\gamma_k}.
\end{equation}
The exponents on the right-hand side can be bounded independently of $k$ since, by the explicit formula \eqref{3.gammak},
\begin{align*}
  \frac{1}{\gamma_k}(2^{k+1}-(k+2)) 
	&= \frac{2^{k+1}-(k+2)}{2^k(\gamma_0 + 1 - \alpha) + \alpha-1}
	\le \frac{2}{\gamma_0+1-\alpha}, \\
	\frac{2^k}{\gamma_k} &= \frac{2^k}{2^k(\gamma_0 + 1 - \alpha) + \alpha-1}
	\le \frac{1}{\gamma_0+1-\alpha}.
\end{align*}
Hence, we can pass to the limit $k\to\infty$ in \eqref{3.Dgammak}, which yields the desired bound for $D$ in $L^\infty(0,T;L^\infty(\Omega))$.
\end{proof}


\section{Weak--strong uniqueness}\label{sec.wsu}

We prove Theorem \ref{thm.wsu}. According to \cite[Lemma 2.4]{LaTz13}, for $0\le m\le\bar{v}\le M$, there exist constants $R>0$ (depending on $m$ and $M$) and $C_1,C_2>0$ (depending on $R$, $m$, and $M$) such that
$$
  h_v(v|\bar{v}) \ge \begin{cases}
  C_1|v-\bar{v}|^2 & \mbox{if }0<v\le R,\ m\le\bar{v}\le M, \\
  C_2|v-\bar{v}|^{\alpha_v} & \mbox{if }v>R,\ m\le\bar{v}\le M,
  \end{cases}
$$
recalling definition \eqref{1.rel} of the relative entropy density. Choosing $0\le v\le M$ and $m\le\bar{v}\le M$, this implies that 
\begin{equation}\label{4.relentL2}
  h_v(v|\bar{v})\ge C|v-\bar{v}|^2, 
  \quad\mbox{where } C=\max\{C_1,C_2(2M)^{\alpha_v-2}\}.
\end{equation}

We compute the time derivative of the relative free energy, defined in \eqref{1.relent}:
\begin{align}\label{4.dHdt}
  \frac{\dd}{\dt}&H[n,p,D|\bar{n},\bar{p},\bar{D}]
  = \sum_{v=n,p,D}\big(\langle\pa_t v,h_v'(v)-h_v'(\bar{v})\rangle
  - \langle\pa_t\bar{v},h''(\bar{v})(v-\bar{v})\rangle\big) \\
  &\phantom{xx}- \lambda^2\langle\pa_t\Delta(V-\bar{V}),
  V-\bar{V}\rangle \nonumber \\
  &= \sum_{v=n,p,D}\big(\langle\pa_t v,h_v'(v)-h_v'(\bar{v})
  -V+\bar{V}\rangle
  - \langle\pa_t\bar{v},h''_v(\bar{v})(v-\bar{v})+V-\bar{V}
  \rangle\big), \nonumber
\end{align}
recalling definition \eqref{1.internal} of the internal energies $h_v$. 
At this point, we need the property $h'_v(v)\in L^2(0,T;H^1(\Omega))$ for $v=n,p,D$, which holds thanks to Lemma \ref{lem.alpha-1}.
We consider the case $v=n$. Inserting the equations 
$\pa_t n=\diver(n\na(h_n'(n)-V))$ and $\pa_t\bar{n}=\diver(\bar{n}\na
(h_n'(\bar{n})-\bar{V}))$ and integrating by parts gives
\begin{align}\label{4.aux}
  \langle \pa_t &n,h_n'(n)-h_n'(\bar{n})-V+\bar{V}\rangle
  - \langle\pa_t\bar{n},h_n''(\bar{n})(n-\bar{n})+V-\bar{V}\rangle \\
  &= -\int_\Omega n\na(h_n'(n)-V)\cdot
  \na\big((h_n'(n)-V)-(h_n'(\bar{n})-\bar{V})\big)\dx \nonumber \\
  &\phantom{xx}- \int_\Omega \bar{n}\na(h_n'(\bar{n})-\bar{V})\cdot\na\big(
  h_n''(\bar{n})(n-\bar{n})+V-\bar{V}\big)\dx \nonumber \\
  &= -\int_\Omega n\big|\na\big((h_n'(n)-V)-(h_n'(\bar{n})-\bar{V})\big)
  \big|^2\dx \nonumber \\
  &\phantom{xx}- \int_\Omega n\na(h'_n(\bar{n})-\bar{V})\cdot\na
  \big((h_n'(n)-V)-(h_n'(\bar{n})-\bar{V})\big)\dx \nonumber \\
  &\phantom{xx}- \int_\Omega \bar{n}\na(h_n'(\bar{n})-\bar{V})\cdot\na\big(
  h_n''(\bar{n})(n-\bar{n})+V-\bar{V}\big)\dx v \nonumber \\
  &= -\int_\Omega n\big|\na\big((h_n'(n)-V)-(h_n'(\bar{n})-\bar{V})\big)
  \big|^2\dx - \int_\Omega\na(h_n'(\bar{n})-\bar{V}) \nonumber \\
  &\phantom{xx}\times
  \big[n\na(h'_n(n)-h'_n(\bar{n})) + \bar{n}\na(
  h_n''(\bar{n})(n-\bar{n})) - (n-\bar{n})\na(V-\bar{V})\big]\dx. 
  \nonumber 
\end{align}
A computation shows that
\begin{align*}
  n\na(h'_n(n)-h'_n(\bar{n})) + \bar{n}\na(h_n''(\bar{n})(n-\bar{n})) 
  = (\alpha_n-1)\na h_n(n|\bar{n}). 
\end{align*}
This identity uses the fact that $h_n$ is given by a power law; it may not hold for general (convex) functions. Taking into account that the first term on the right-hand side of \eqref{4.aux} is nonpositive, we obtain, after integrating by parts (observe that $h_n(n|\bar{n})=0$ on $\Gamma_{\mathrm{Dir}}$) and using Young's inequality,
\begin{align*}
  \langle \pa_t &n,h_n'(n)-h_n'(\bar{n})-V+\bar{V}\rangle
  - \langle\pa_t\bar{n},h_n''(\bar{n})(n-\bar{n})+V-\bar{V}\rangle \\
  &\le -(\alpha_n-1)\int_\Omega\na(h'_n(\bar{n})-\bar{V})\cdot
  \na h_n(n|\bar{n})\dx \\
  &\phantom{xx}+ \int_\Omega\na(h'_n(\bar{n})-\bar{V})
  \cdot\na(V-\bar{V})(n-\bar{n})\dx \\
  &\le (\alpha_n-1)\|\Delta(h'_n(\bar{n})-\bar{V})\|_{L^\infty(\Omega)}
  \int_\Omega h_n(n|\bar{n})\dx \\
  &\phantom{xx}+ \|\na(h'_n(\bar{n})-\bar{V})\|_{L^\infty(\Omega)}
  \|\na(V-\bar{V})\|_{L^2(\Omega)}\|n-\bar{n}\|_{L^2(\Omega)}.
\end{align*}
By assumption, $h'_n(\bar{n})-\bar{V}$ is bounded in $L^\infty(0,T;W^{2,\infty}(\Omega))$. Therefore,
\begin{align*}
  \langle \pa_t &n,h_n'(n)-h_n'(\bar{n})-V+\bar{V}\rangle
    - \langle\pa_t\bar{n},h_n''(\bar{n})(n-\bar{n})+V-\bar{V}\rangle \\
    &\le C\int_\Omega h_n(n|\bar{n})\dx 
    + C\|\na(V-\bar{V})\|_{L^2(\Omega)}^2
    + C\|n-\bar{n}\|_{L^2(\Omega)}^2.
\end{align*}
We deduce from inequality \eqref{4.relentL2} that the last term is bounded according to 
$$
  \|n-\bar{n}\|_{L^2(\Omega)}^2\le \int_\Omega h_n(n|\bar{n})\dx.
$$
Similar estimates are derived for $p$ and $D$. Summarizing, we conclude from \eqref{4.dHdt} that
\begin{align*}
  \frac{\dd}{\dt}&H[n,p,D|\bar{n},\bar{p},\bar{D}]
  \le C\int_\Omega H[n,p,D|\bar{n},\bar{p},\bar{D}]\dx.
\end{align*}
Gronwall's inequality and the fact that $H(n,p,D|\bar{n},\bar{p},\bar{D}) = 0$ at $t=0$ finish the proof.


\end{document}